\documentclass{llncs}

\usepackage{fullpage}

\usepackage[utf8]{inputenc}

\usepackage{tikz}
\usetikzlibrary{math,calc,arrows,shapes,positioning,decorations.pathreplacing}

\usepackage[english]{babel}
\usepackage{amssymb,mathtools}
\usepackage{dsfont}
\usepackage{enumerate}
\usepackage{subcaption}
\usepackage[linesnumbered,lined,boxruled]{algorithm2e}

\usepackage{natbib}
\bibpunct[, ]{(}{)}{,}{a}{}{,}%
\def\bibfont{\small}%

\SetAlCapSkip{0.5em}

\usepackage[title]{appendix}

\newcommand{\abs}[1]{\vert #1 \vert}
\newcommand{\set}[1]{\{ #1 \}}
\newcommand{\norm}[1]{\| #1 \|}

\newcommand{\conv}{\operatorname{conv}}

\usepackage[textwidth=20mm,obeyFinal,colorinlistoftodos]{todonotes}

\DeclareUnicodeCharacter{00F6}{\"o}

\usepackage[colorlinks=true,breaklinks=true,bookmarks=false,urlcolor=blue,citecolor=blue,linkcolor=blue,bookmarksopen=false,draft=false]{hyperref}

\usetikzlibrary{decorations.pathreplacing}


\begin{document}

\title{An Algorithm for the Separation-Preserving\\ Transition of Clusterings}

\author{Steffen Borgwardt\inst{1} \and Felix Happach\inst{2} \and Stetson Zirkelbach\inst{3}}

\institute{\email{\href{mailto:steffen.borgwardt@ucdenver.edu}{steffen.borgwardt@ucdenver.edu}};
University of Colorado Denver 
\and \email{\href{mailto:felix.happach@tum.de}{felix.happach@tum.de}}; Technical University of Munich
\and \email{\href{mailto:stetson.zirkelbach@ucdenver.edu}{stetson.zirkelbach@ucdenver.edu}};
University of Colorado Denver 
}

\date{}

\maketitle

\begin{abstract}
The separability of clusters is one of the most desired properties in clustering. There is a wide range of settings in which different clusterings of the same data set appear. We are interested in applications where there is a need for an explicit, gradual transition of one separable clustering into another one. This transition should be a sequence of simple, natural steps that upholds separability of the clusters throughout.

We design an algorithm for such a transition. We exploit the intimate connection of separability and linear programming over bounded-shape partition and transportation polytopes: separable clusterings lie on the boundary of partition polytopes, form a subset of the vertices of the corresponding transportation polytopes, and circuits of both polytopes are readily interpreted as sequential or cyclical exchanges of items between clusters. This allows for a natural approach to achieve the desired transition through a combination of two walks: an edge walk between two so-called radial clusterings in a transportation polytope, computed through an adaptation of classical tools of sensitivity analysis and parametric programming; and a walk from a separable clustering to a corresponding radial clustering, computed through a tailored, iterative routine updating cluster sizes and re-optimizing the cluster assignment of items.
\end{abstract}
 

\section{Introduction}\label{sec:intro}

The {\em partitioning} or {\em clustering} of a data set $X = \{x_1,\dots,x_n\}$ into a set $C = (C_1,\dots,C_k)$ of disjoint clusters is arguably the most important task in unsupervised learning. It arises in applications in many fields, including operations research, machine learning, and statistics, and there is a wealth of literature on theory, algorithms, and applications \citep{v-98,jmf-99,ss-02,xw-08,bdw-09,hr-12,ar-13}. 

Our interest lies not in the task itself, but an application derived thereof: given two so-called separable clusterings, we design an algorithm to compute a gradual transition between them, in the form of a sequence of clusterings, retaining separability throughout. We begin with some formal definitions and notation about clusterings and separability, before turning to motivation of this work and an outline of our contributions.

\subsection{Separable Clusterings}

Let us provide a formal definition of the term clustering and some related terms.

\begin{definition}[Clustering]\label{def:clustering}
A \emph{clustering} of $X$ is a partition $C = (C_1,\dots,C_k)$ of $X$. For $i \in [k]$, we call $C_i$ the \emph{$i$-th cluster} and $\abs{C_i}$ denotes its \emph{size}, i.e., the number of items in $C_i$.
The \emph{shape} of $C$ is $\abs{C} \coloneqq (\abs{C_1},\dots,\abs{C_k})$.
\end{definition}


Based on a similarity measure for items in the data set, the goal of clustering is to assign items that are similar to each other to the same cluster. Most clustering tasks are performed in a geometric setting: the data set $X = \{x_1,\dots,x_n\} \subseteq \mathbb{R}^d$ is represented as a finite collection of distinct items (i.e., $x_i\neq x_j$ for $i\neq j$) in a $d$-dimensional real space, and measures of similarity correspond to a norm or quasi-norm in the space; the Euclidean norm $\norm{\cdot}$ or squared Euclidean norm $\norm{\cdot}^2$ are common. More sophisticated similarity measures are typically implemented indirectly, through the application of a kernel function to transform the space $\mathbb{R}^d$ underlying $X$ to a different one in which the Euclidean norm or its square again are viable measures \citep{ss-02}. In this paper, we assume that the items in $X$ are distinct and use the squared Euclidean norm.

In this setting, the desire to create clusters of similar items becomes a desire to partition the data set into {\em separable} clusters. 
Two clusters $C_1,C_2$ are called separable if there exists a (separating) hyperplane that partitions the underlying space into two halfspaces, each of which contains one of the clusters. A {\em separable clustering} $C = (C_1,\dots,C_k)$ requires separability of all pairs of clusters, as well as a special positioning of the corresponding hyperplanes: they have to create a partition of the underlying space into a cell complex $P = (P_1,\dots,P_k)$ of polyhedral cells, one cell $P_i$ for each cluster $C_i$ formed through the intersection of all halfspaces that contain $C_i$. 

Separable clusterings have an algebraic representation as {\em (constrained) least-squares assignments}. Let $s_1,\dots,s_k \in \mathbb{R}^d$ be a set of {\em sites} (or centers) in the same space as data set $X$. For a simple wording, we call the collection of all sites, as well as the vector $s=(s_1^T,\dots,s_k^T)^T$, a {\em site vector}.


\begin{definition}[Least-Squares Assignment (LSA)]\label{def:LSA}
A \emph{least-squares assignment}, or LSA, for a given site vector $s_1,\dots,s_k \in \mathbb{R}^d$ is a clustering $C = (C_1,\dots,C_k)$ that minimizes $$\sum_{i = 1}^k \sum_{x \in C_i} \norm{x - s_i}^2.$$
A clustering that minimizes this term over all clusterings with the same shape as $C$ is called a \emph{constrained (or balanced) LSA}. \end{definition}

Most LSAs considered in this paper are constrained LSAs, and we simply call them LSAs when the context is clear. In this paper, we assume that $\sum_{j = 1}^n x_j = 0$. Note that if $v=\sum_{j = 1}^n x_j \neq 0$, the whole data set $X$ can be translated by $-\frac{v}{n}$ to have this property. This is not a restriction for the computation of an LSA, which can be seen by translating the site vector $s$ in the same way.


Constrained LSAs are intimately connected to cell complexes called \emph{power diagrams}. Power diagrams are a classical topic in computational geometry, and generalize the well-known Voronoi diagrams \citep{Aurenhammer1987,AurenhammerHoffmannAronov1998}. A power diagram can be represented in several ways and it is trivial to switch between representations \citep{Aurenhammer1987,Borgwardt2010}. For example, a power diagram can be specified through the definition of a ball, with site (or center) $s_i$ and radius $r_i$, for each cell $P_i$. The hyperplanes separating pairs of cells are equally far from the sites with respect to a distance measure based on the corresponding balls. Geometrically, they run through the common intersection points of (a joint scaling of) the balls. See Figure \ref{fig:pd} for an example of this construction. In this paper, we use an alternative, equivalent representation in which the positions of hyperplanes separating the cells are given explicitly through differences of values  $\gamma_1,\dots,\gamma_k \in \mathbb{R}$; see~\cite{Aurenhammer1987,Borgwardt2015}. We provide a formal definition.


\begin{figure}
\centering
\includegraphics[scale=0.6]{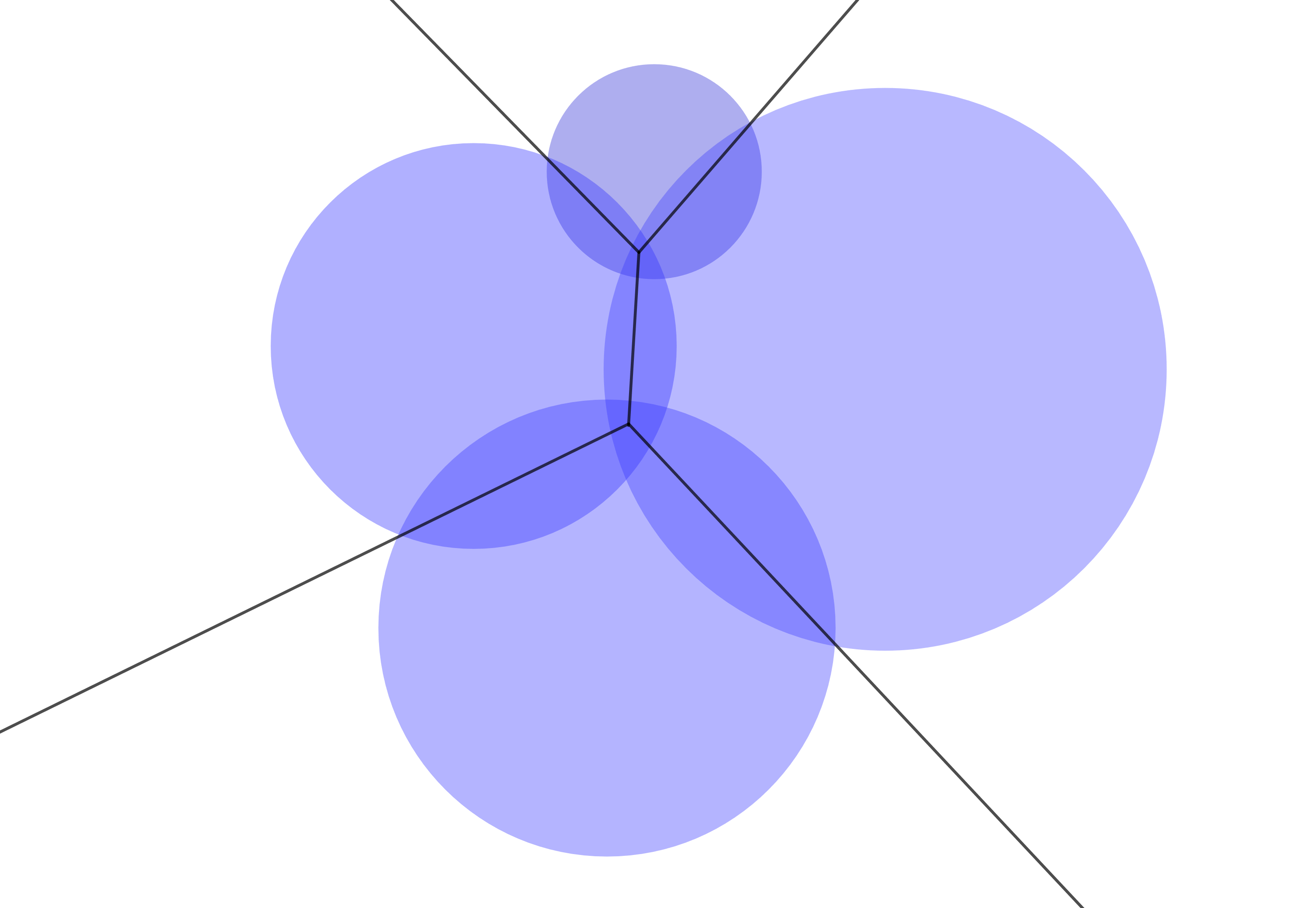}
\caption{A power diagram in $\mathbb{R}^2$. Each cell specifies a site and radius of a ball around it. Separating hyperplanes between cells run through the common intersection points of a joint scaling of the corresponding balls.}\label{fig:pd}
\end{figure}

\begin{definition}[Power Diagram]\label{def:powerdiagram}
For given $\gamma_1,\dots,\gamma_k \in \mathbb{R}$ and site vector $s_1,\dots,s_k \in \mathbb{R}^d$,  the set
\begin{equation}\label{powerdiagram:hyperplane}
P_i \coloneqq \{ x \in \mathbb{R}^d \, | \, (s_\ell - s_i)^T x \leq \gamma_\ell - \gamma_i \, \text{ for all } \ell \in [k] \setminus \{i\} \},
\end{equation}
is the $i$-th cell of the \emph{power diagram} $(P_1,\dots,P_k)$ of $\mathbb{R}^d$.
\end{definition}


\cite{AurenhammerHoffmannAronov1998} proved the following connection between power diagrams and constrained LSAs: if a power diagram satisfies $C_i \subseteq P_i$ for all $i \in [k]$, then $C$ is a constrained LSA to the site vector $s$ of the power diagram. We call a power diagram that satisfies this property a {\em separating power diagram} for the underlying LSA. Conversely, if $C$ is a constrained LSA to a given site vector $s$, then there exists a power diagram with site vector $s$ that \emph{induces} $C$ by assigning all items in the same cell to the same cluster (and for items on the boundary between cells to any of these cells' clusters). See Figure \ref{fig:basic} for an example of a separating power diagram. The same concept of cluster separation appears under other names in the literature, such as multiclass support vector machines \citep{bb-99,v-98,ww-99,cs-01} or piecewise-linear separability \citep{bm-92}.

\begin{figure}[t]
\centering
\includegraphics[scale=1.0]{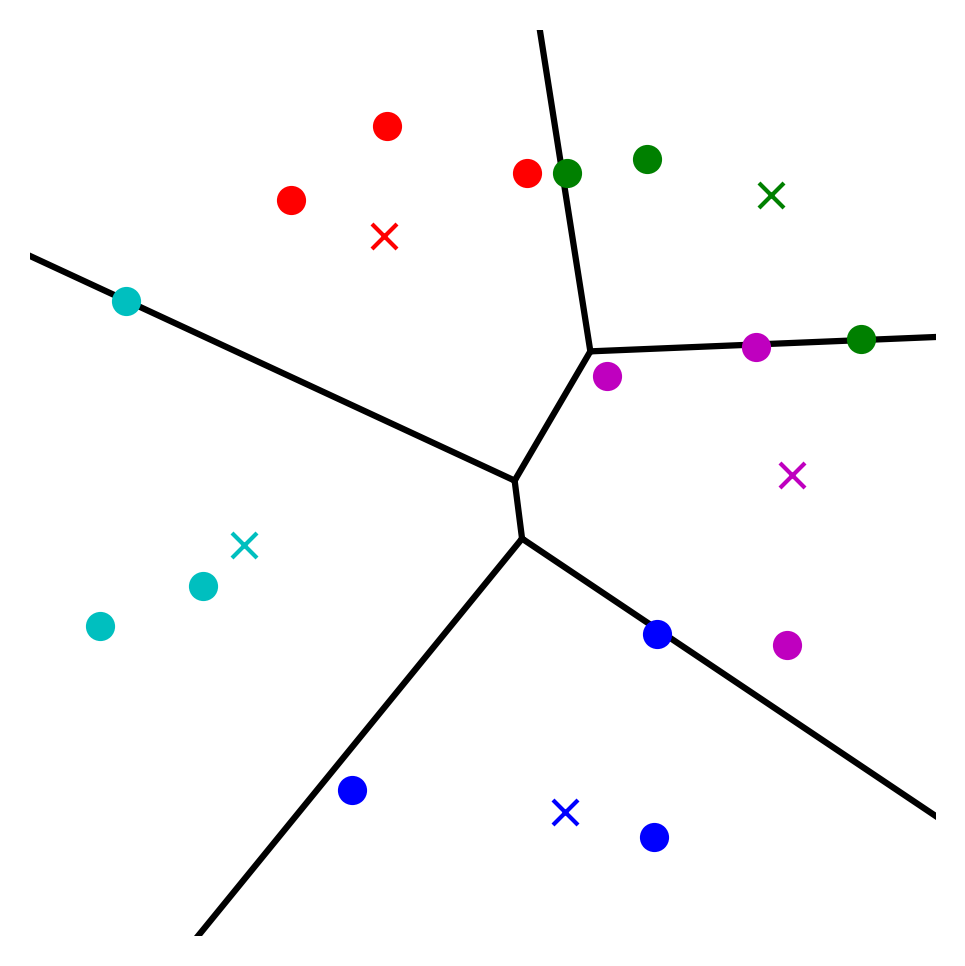}
\caption{A separating power diagram for a clustering of five clusters (different colors) in $\mathbb{R}^2$. Items of the clusters are shown as filled circles, the sites for the cells are shown as crosses.}
\label{fig:basic}
\end{figure}

Informally, there is a one-to-many correspondence between (separable) clusterings that `allow' separating power diagrams, and power diagrams that `induce' the clustering by assigning all items in a cell to the same cluster \citep{AurenhammerHoffmannAronov1998}. Importantly, a constrained LSA and corresponding separating power diagram can be constructed from the {\em same} site vectors $s_1,\dots,s_k$ (marked as crosses in Figure \ref{fig:basic} and later figures). 

When two clusterings are equally good with respect to a given site vector, then there exists a power diagram that serves as a separating power diagram for both clusterings; see Figure \ref{fig:SharedSepPD}. We call such a power diagram a {\em shared (separating) power diagram}. Note that clusterings with a shared power diagram can only differ by items that lie on the separating hyperplanes.

\begin{figure}[t]
\centering
\subcaptionbox{First clustering}{\includegraphics[width=0.45\textwidth]{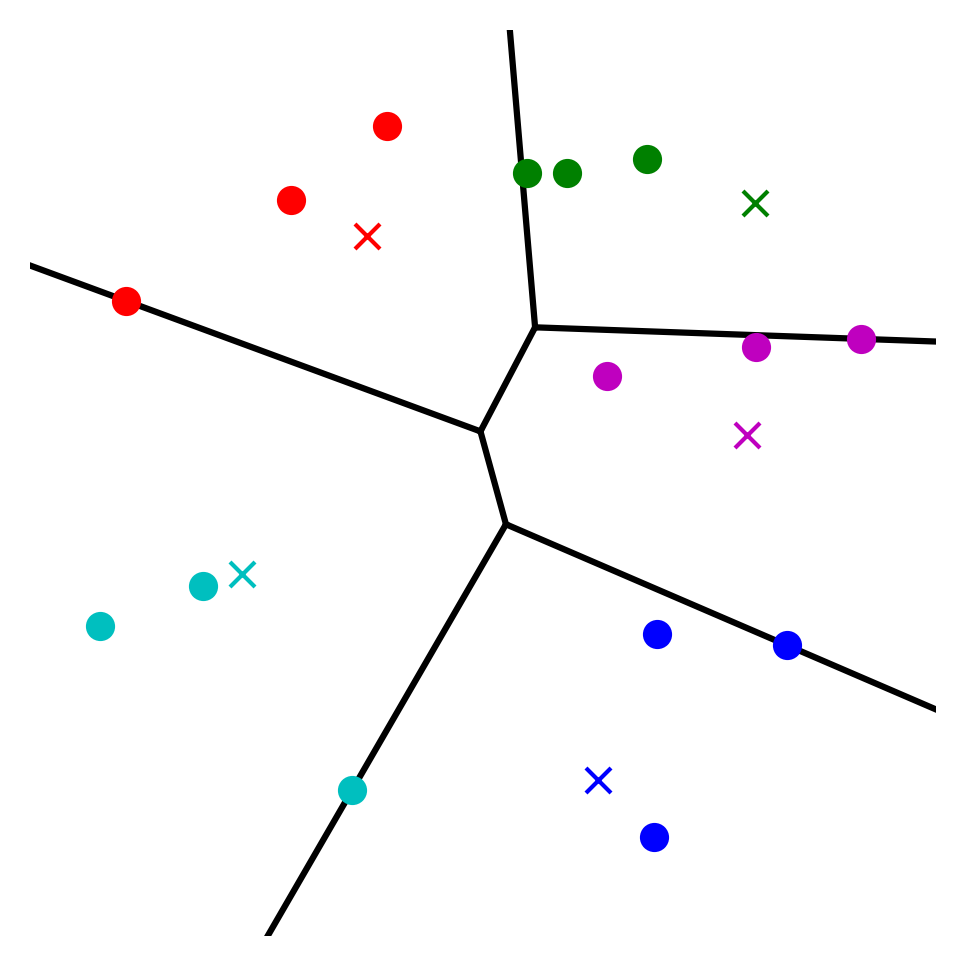}}%
\hfill
\subcaptionbox{Second clustering }{\includegraphics[width=0.45\textwidth]{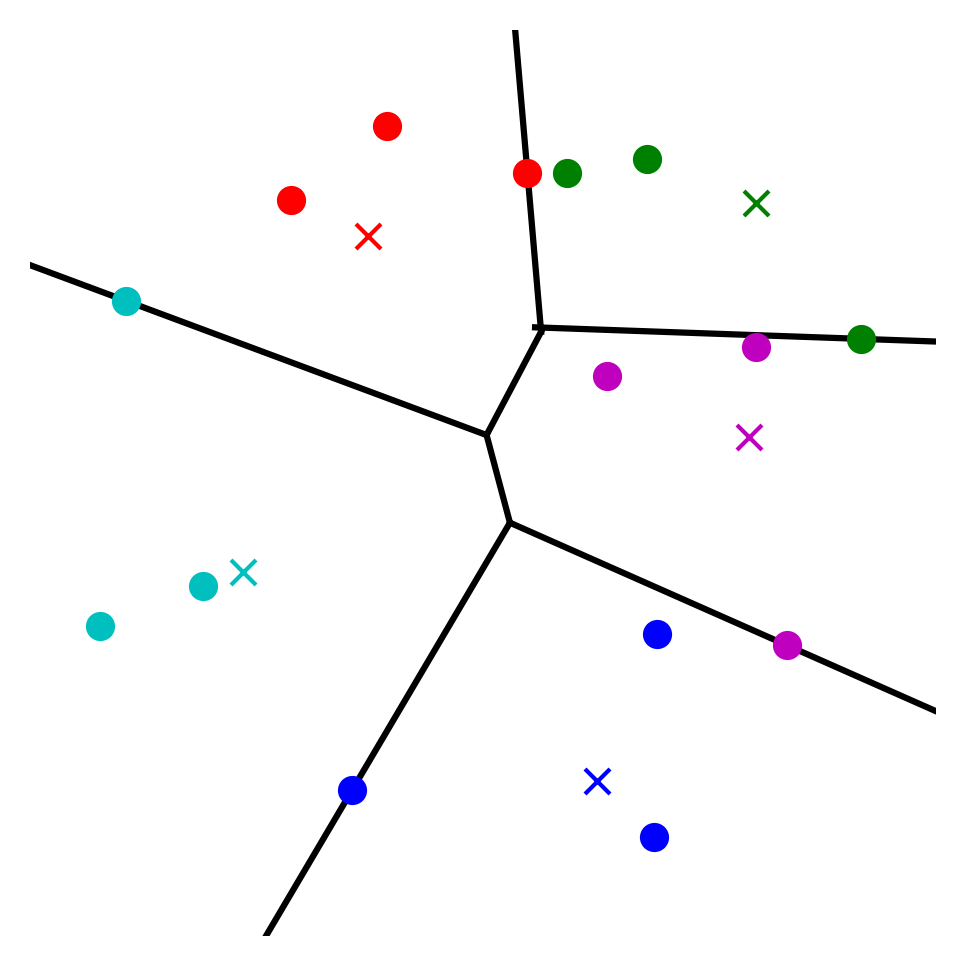}}%
\caption{Two clusterings that allow a shared separating power diagram. The clusterings only differ by items that lie on the separating hyperplanes between clusters.}
\label{fig:SharedSepPD}
\end{figure}

\subsection{Transitions between Separable Clusterings}\label{sec:why}

It is common for different separable clusterings of the same data set to become relevant. For example, different algorithms for the clustering of the same data set return different solutions. Popular algorithms like the $k$-means algorithm return different solutions themselves, depending on their input; in the case of $k$-means, depending on the inital sites \citep{m-67}. It is standard practice to run several algorithms to find several clusterings of the same data set, and then to choose one that exhibits desirable properties.

There is considerable interest in comparing clusterings and measuring their similarity; see, e.g., \cite{ma-84,ww-07,m-07,g-17}. Given two clusterings, such measures are typically built from pairwise relationships among the data items, such as a ratio of the number of items that were assigned to the same cluster versus to different clusters. In particular, this is done to measure the `robustness' of a clustering -- if different algorithms return similar clusterings, it is likely that a natural structure in the data set has been found. Observations on such clusterings are deemed more reliable. 

In this paper, we improve on a recent direction of research in which different clusterings of the same data set appear in a different way: in some applications, it is of interest to design an {\em explicit transition} between two clusterings. Our goal is to design algorithms for the transition between two separable clusterings of the same data set; we call them {\em initial} clustering and {\em target} clustering. Figure \ref{fig:problemVis} shows an example. The transition should be a sequence of clusterings, each of them {\em retaining separability}, that gradually transitions the initial clustering into the target one. Our key contribution over previous work is the retention of separability. As previous work \citep{Borgwardt2013,bv-19a} is not designed to take into account locations of items in $\mathbb{R}^d$, new methodology is required to compute such a transition. We develop it in this paper.
%

\begin{figure}
\centering
\subcaptionbox{Initial clustering}{\includegraphics[width=0.45\textwidth]{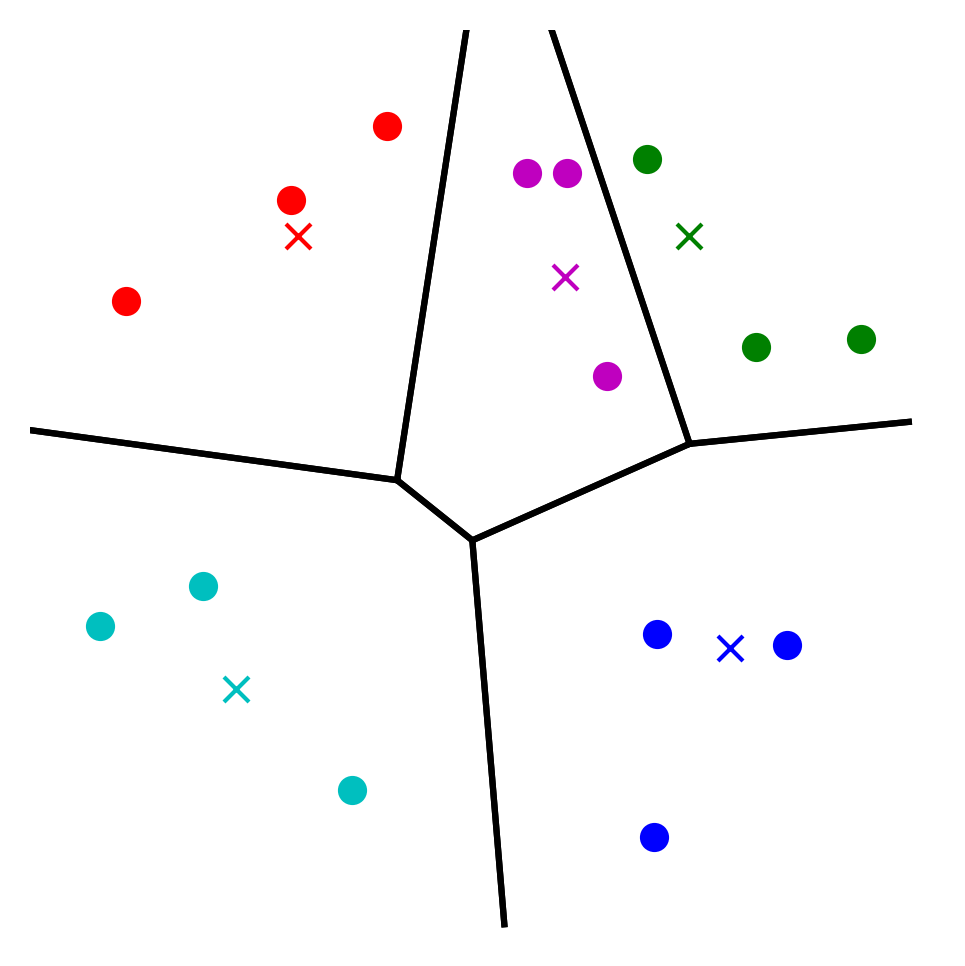}}%
\hfill
\subcaptionbox{Target clustering }{\includegraphics[width=0.45\textwidth]{5ptSimpleEndLSA.png}}%
\caption{An initial and target clustering of the same data set. We design an algorithm to find a sequence of separable clusterings that gradually transitions the initial clustering into the target clustering.}
\label{fig:problemVis}
\end{figure}

The ability to compute such transitions facilitates a number of applications. For example, it gives a conceptually different measure for the similarity of two clusterings in the form of a {\em transformation distance} \citep{bv-19a}. For this work, the most important driver is a common situation in operations research: there is an (original) initial clustering of a data set reflecting the current state, as well as a (new, desired) target clustering that should be implemented in practice. The two clusterings may differ significantly, and one wants to devise an explicit, gradual transition between them over time. Intermediate clusterings are temporary solutions in practice. 

Our original interest in a transition between separable clusterings came from an application in land consolidation. The redistribution of farmland can be modeled as a clustering problem \citep{b-03,bbg-11}: each cluster represents a farmer and the items are the geographical locations of lots in an agricultural region. Separable clusterings are desirable due to placing each farmer's lots in a convex cell -- in practice, this translates to adjacent lots being assigned to the same farmer, low overall driving distances, and a dramatically more efficient cultivation overall \citep{bbg-14}. Over the last two decades, a number of such land redistributions have been completed successfully in Southern Germany and the regions transitioned to a separable clustering. However, the planned redistributions would often trade more than $50\%$ of lots. For stability of their farming processes, the farmers asked for a gradual transition to the target clustering through a sequence of intermediate clusterings \citep{bv-19a}. 

Such land redistributions are implemented through so-called lend-lease agreements that run for five to ten years. During the long period of such an agreement, properties of lots and farmers' possession in the region change, and a new separable clustering is planned for the next time-period. The goal becomes to gradually transition from a separable clustering to the new separable clustering. Keeping separability throughout this transition means that an efficient cultivation always remains possible; this has become one of the main requests of the farmers.

We see similar applications in high-efficiency, large-scale computing. It relies on the scheduling and grouping of (computation) jobs to computing resources. The scheduling of jobs to resources has become increasingly complex and there is a wide range of commercial solutions \citep{rbabbhjmprk-16}. Computing resources are commonly represented as items of a data set in two- or three-dimensional space, with physical distances corresponding to the communication times between the resources.  For efficiency, the resources assigned to the same process should be close to each other \citep{ucbvp-19}. This naturally leads to an integration of clustering techniques -- the aforementioned `grouping' -- into the scheduling \citep{lb-21}; each job is a cluster, and separable clusterings correspond to an efficient assignment of resources. As jobs start and end at different times and request different amounts of resources at different times -- freeing unneeded resources or requesting more -- schedulers provide regular updates to the current clustering. 

 Similar questions also arise in the setting of data storage, where related information (a cluster) is stored in resources that are physically close to each other for fast access; again separable clusterings are particularly efficient. The necessity to make changes to current storage is an important consideration for state-of-the art storage solutions \citep{amazon-16}. For example, a planned maintenance of the resources, such as rebooting or physical replacement, or the migration to a new storage system would lead to the desire to transition to a new clustering. Retaining separability throughout would guarantee fast access during the transition.




There are further promising applications in other areas, such as in the clustering of customers in service, entertainment, and insurance industries. A company may want to gradually transition their customers to a new clustering of premium classes over time. Separable clusterings of customers are fair in the sense that there are no outliers: the underlying separating power diagram serves as a classifier; new customers are assigned to the cell's cluster that they fall in \citep{bm-92,v-98}. The ability to preserve separability of the clusterings during a gradual transition retains the ability to fairly assign customers to a cluster at any time.


\subsection{Outline}\label{sec:outline}
In this paper, we provide the theory and develop methods for a gradual transition between two separable clusterings in the form of a sequence of clusterings, each retaining separability. 
In Section \ref{sec:contribution}, we provide an overview of our strategy, and introduce necessary tools from the literature. In Section \ref{sec:transition:overview}, we describe how to algorithmically realize this strategy and discuss the properties of the transition. Sections \ref{sec:init-to-rad} and \ref{sec:rad-to-rad} are dedicated to the necessary technical details and proofs. 
We conclude with a brief outlook on some open questions and natural next steps, in Section \ref{sec:conclusion}. Proof-of-concept implementations and some examples are available at \url{https://github.com/szirkelbach/Transitioning-Separable-Clusterings}. We provide a couple of brief appendices. In Appendix \ref{app:ranging}, we recall some background on ranging for degenerate vertices; these techniques are required for an implementation of our algorithms. In Appendix \ref{app:radial}, we take a closer look at the (restriction of) cluster sizes throughout the transition. In Appendix \ref{app:computations}, we report on some computational experiments.


\section{The Main Strategy and Important Tools}\label{sec:contribution}


A key goal in the design of a gradual transition between two separable clusterings is for it to take only few, simple steps. This drives the design of our approach in two ways, which we describe in Section \ref{sec:simpletransitions}. In Section \ref{sec:upholdingseparability}, we exhibit our tools to keep separability throughout the transition.

\subsection{Simple Steps, Direct Transitions}\label{sec:simpletransitions}

First, the transition should be a sequence of simple steps. We achieve this through the application of a so-called {\em sequential} or {\em cyclical exchange} of items in each step. For such an exchange, one selects an ordered list of clusters and one item from each cluster; the next clustering in the transition is then derived by moving the items to new clusters following the ordered list. These exchanges are a simple choice that allows for the construction of the desired transition, and they have been important in the construction of clustering transitions in easier settings and the studies of combinatorial diameters \citep{Borgwardt2013,bv-19a}. 

To compare two clusterings and to formally define exchanges of items, we make use of a so-called \emph{clustering difference graph}. For two clusterings $C,C'$, the clustering difference graph $CDG(C,C')$ is a labeled directed multigraph with one node for every $i \in [k]$ and an arc $(i,\ell)$ for distinct $i,\ell \in [k]$ with label $x$ if $x \in C_i \cap C'_\ell$.

\begin{definition}[Clustering Difference Graph]\label{def:CDG}
The clustering difference graph (CDG) of two clusterings $C,C'$ of $X$ is defined as $CDG(C,C') = ([k],E)$ with $E = \set{(i,\ell,x) \, | \, i,\ell \in [k], \, i\not= \ell, \, x \in C_i \cap C'_\ell}$.
\end{definition}

For our purposes, CDGs appear when we compare a current clustering $C$ to a target clustering $C'$. The $CDG(C,C')$ is a convenient way to represent the necessary changes in the assignment of items to clusters. Informally, a CDG contains an arc from node $i$ to node $j$ with label $x$ if the item $x$ has to move from cluster $C_i$ to cluster $C'_j$. One may remove isolated nodes in a CDG; these nodes correspond to clusters that have the same items in both $C$ and $C'$.

We call a cycle in a CDG a \emph{cyclical exchange} of items between clusters, and a path in the clustering graph a \emph{sequential exchange}. If $CDG(C,C')$ consists of a single cyclical or sequential exchange, we say that $C$ and $C'$ {\em differ} by a single exchange. {\em Applying} a cyclical or sequential exchange refers to updating $C$ to the clustering $C'$ that differs from $C$ by only this exchange. Informally, the item reassignments indicated by the arcs of the cyclical or sequential exchange are performed. Figure \ref{fig:basisChange} depicts an example.

\begin{figure}[t]
\centering
\subcaptionbox{Initial clustering and associated power diagram}
{\includegraphics[width=0.35\textwidth]{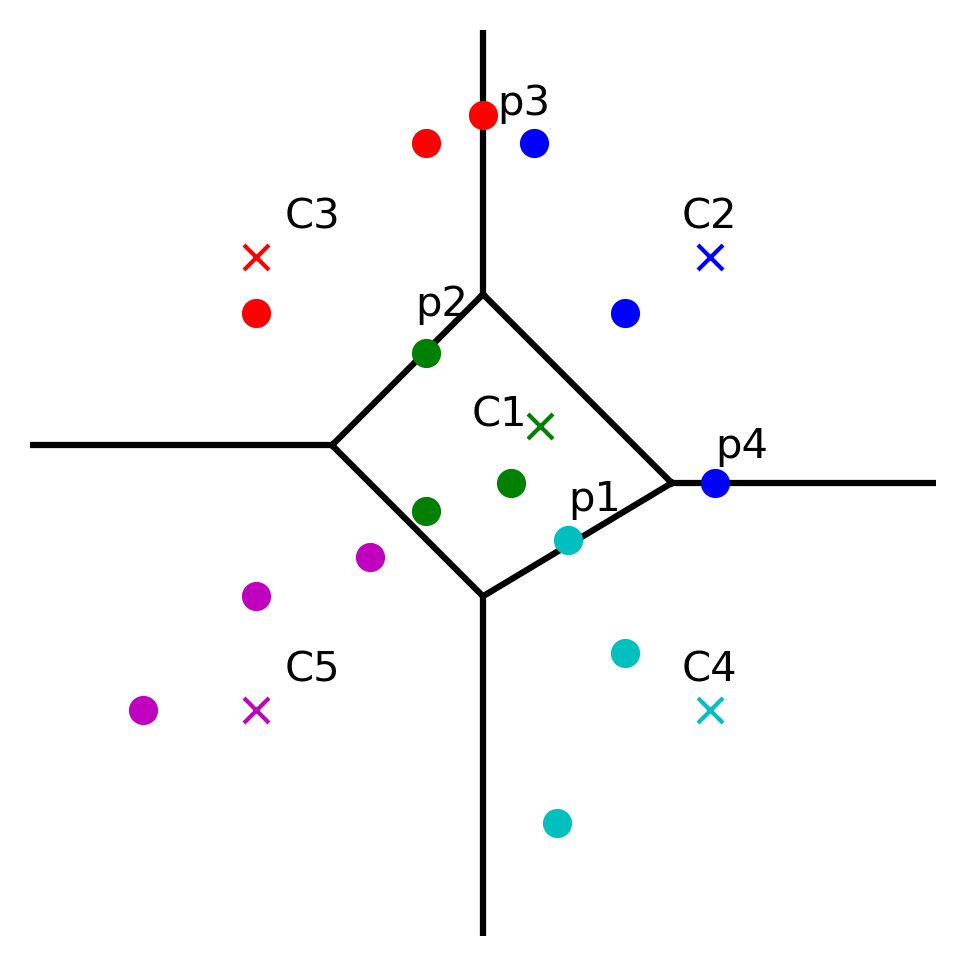}}%
\hfill
\subcaptionbox{Target clustering that allows the same, shared power diagram}
{\includegraphics[width=0.35\textwidth]{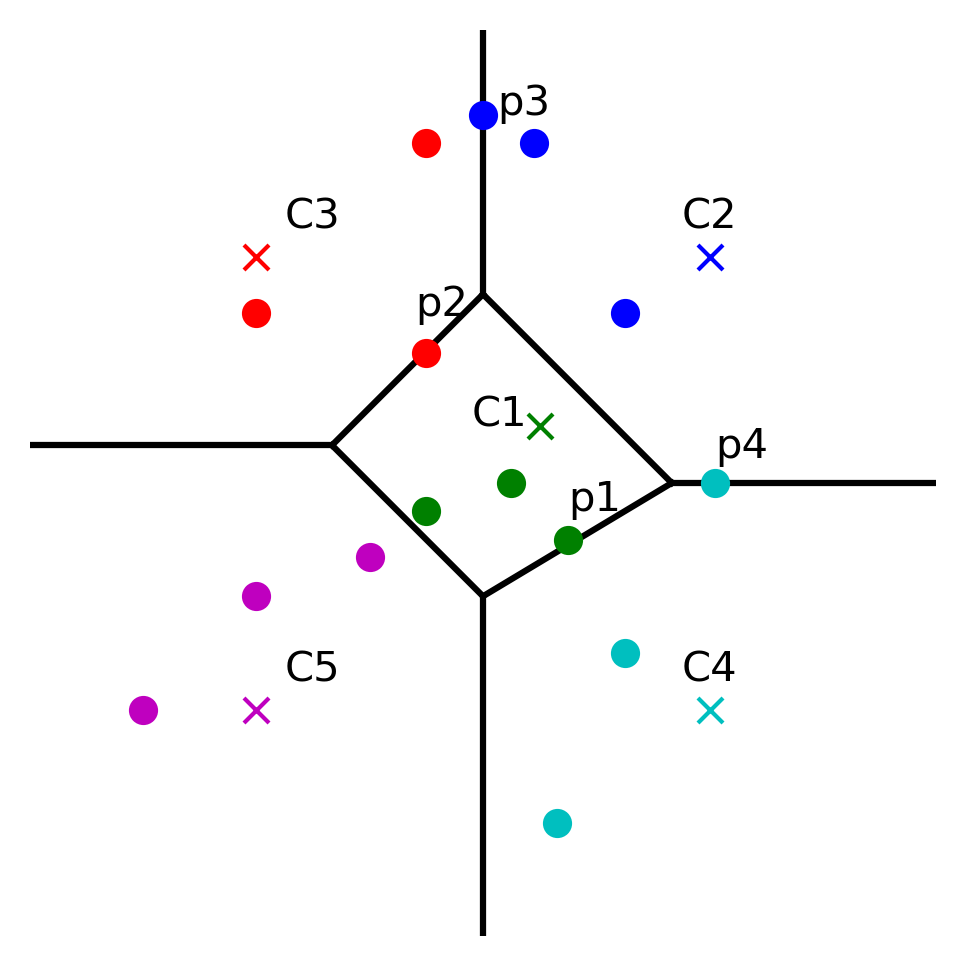}}%
\hfill
\subcaptionbox{The cyclical exchange required to transition the initial into the target clustering.}
{\includegraphics[width=0.45\textwidth]{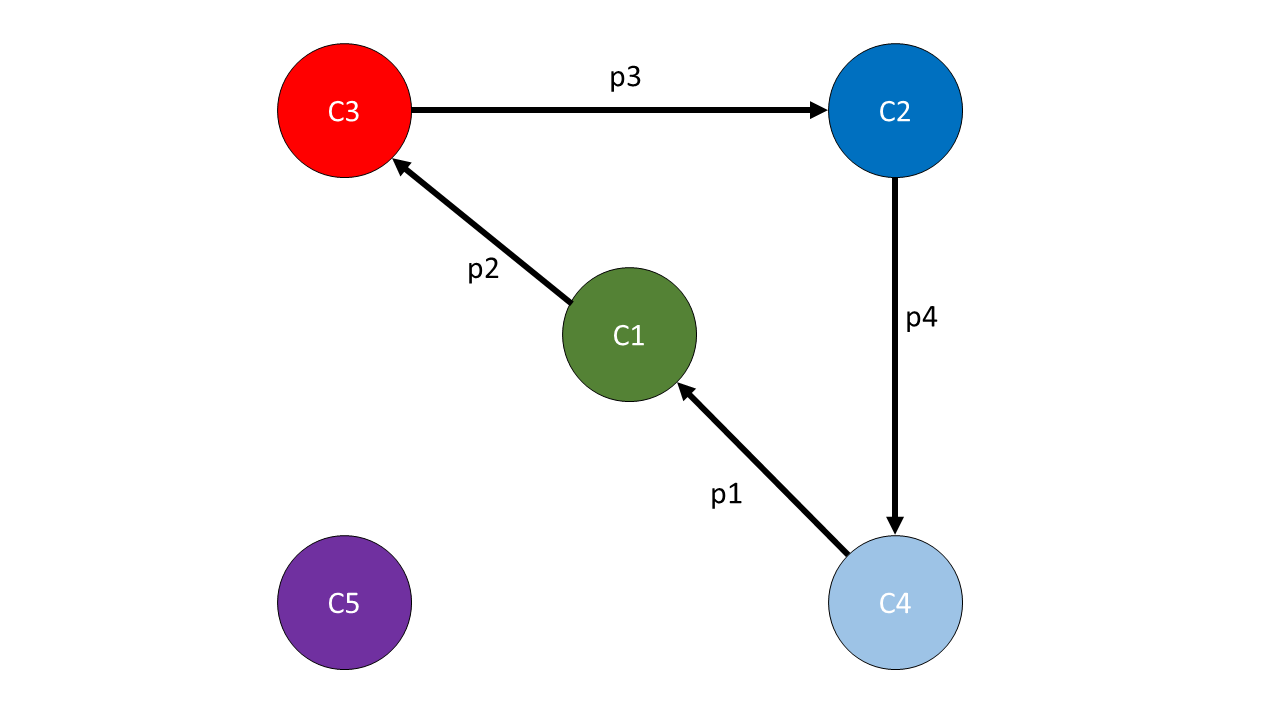}}
\caption{Two clusterings (top) that allow a shared power diagram. The two clusterings differ only by a single cyclical exchange (bottom): the CDG shows clusters as nodes; the arcs between them are labeled with the items that have to be moved. }
\label{fig:basisChange}
\end{figure}

For a transition between two clusterings, we create a sequence of intermediate clusterings that differ from the next by just a single exchange. There are several arguments why this is a natural choice: the only conceptually simpler difference between two clusterings takes the form of an exchange of a {\em single item}. However, such simpler exchanges would immediately violate cluster size bounds if the shapes of $C$ and $C'$ coincide -- in this case, $CDG(C,C')$ decomposes into arc-disjoint cycles, and only cyclical exchanges may be applied to retain feasible clusterings. Further, our more general exchanges, in a sense, aggregate several single-item exchanges, and perform them together. This is beneficial for a low number of steps, i.e., number of clusterings, in the transition. 

Second, the sites used for the construction of both LSAs and separating power diagrams take a central role. Let $s=(s_1^T,\dots,s_k^T)^T$ be the site vector for the initial clustering $C^s$ in the transition, and let $t=(t_1^T,\dots,t_k^T)^T$ be the site vector for the target clustering $C^t$. We construct a sequence of clusterings that follows a {\em linear transition} from $s$ to $t$: essentially, we want to identify the intermediate clusterings and power diagrams that occur if we linearly move the location of the sites $s_1,\dots,s_k$ of $C^s$ to the sites $t_1,\dots,t_k$ of $C^t$. Each of these clusterings is itself a constrained LSA for a site vector $(1-\lambda) s +  \lambda t$ for some $0<\lambda<1$. 

Again, there are several arguments why a linear transition from $s$ to $t$ is a natural choice. It corresponds to the shortest transition between the site vectors in a Euclidean-distance sense. In turn, this is beneficial to the number of steps in the transition. Each intermediate clustering stays related to $s$ and $t$ through the convex construction of its site vector from $s$ and $t$ and some $\lambda \in (0,1)$. In particular, $\lambda$ can be used as a percentage to measure the current progress of the transition. 

Figure \ref{fig:fulltransition} shows a full transition between two LSAs, as computed using the methods in this paper. Each consecutive pair of clusterings differs only by a single exchange of items. Subfigures $(b)$ to $(e)$ follow a linear transition of the sites. (The values $\lambda=\frac{1}{2}$ and $\frac{2}{3}$ are rounded for a clean presentation.) The steps from subfigures $(a)$ to $(b)$ and from $(e)$ to $(f)$ retain the initial and target set of sites. These steps are a form of pre-processing to facilitate the linear transition of sites. We provide a high-level overview of our algorithms in Section \ref{sec:transition:overview}.
\begin{figure}
\centering
\subcaptionbox{Initial LSA (sites $s$)}{\includegraphics[width=0.4\textwidth]{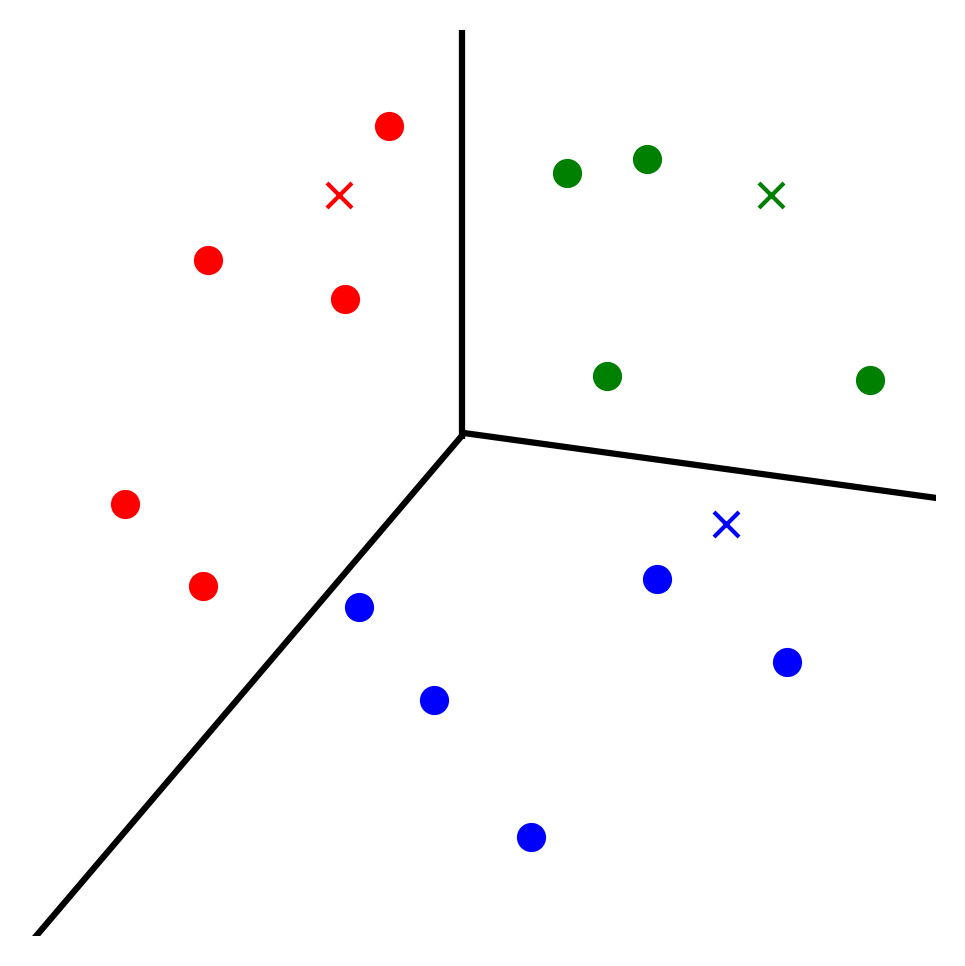} }%
\hfill
\subcaptionbox{Start of linear transition (sites $s$)}{\includegraphics[width=0.4\textwidth]{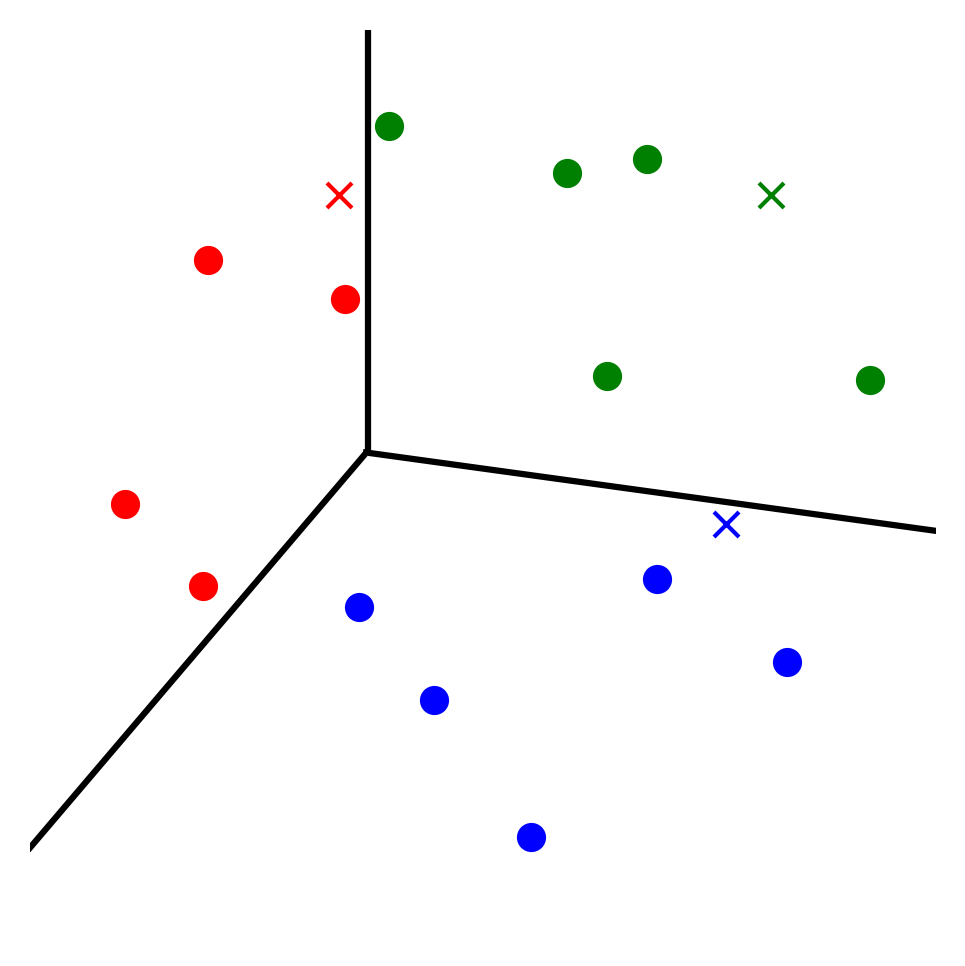}}%

\subcaptionbox{Step of linear transition \\\hspace*{0.66cm}(sites $(1-\frac{1}{2}) s +  \frac{1}{2} t$)}{\includegraphics[width=0.4\textwidth]{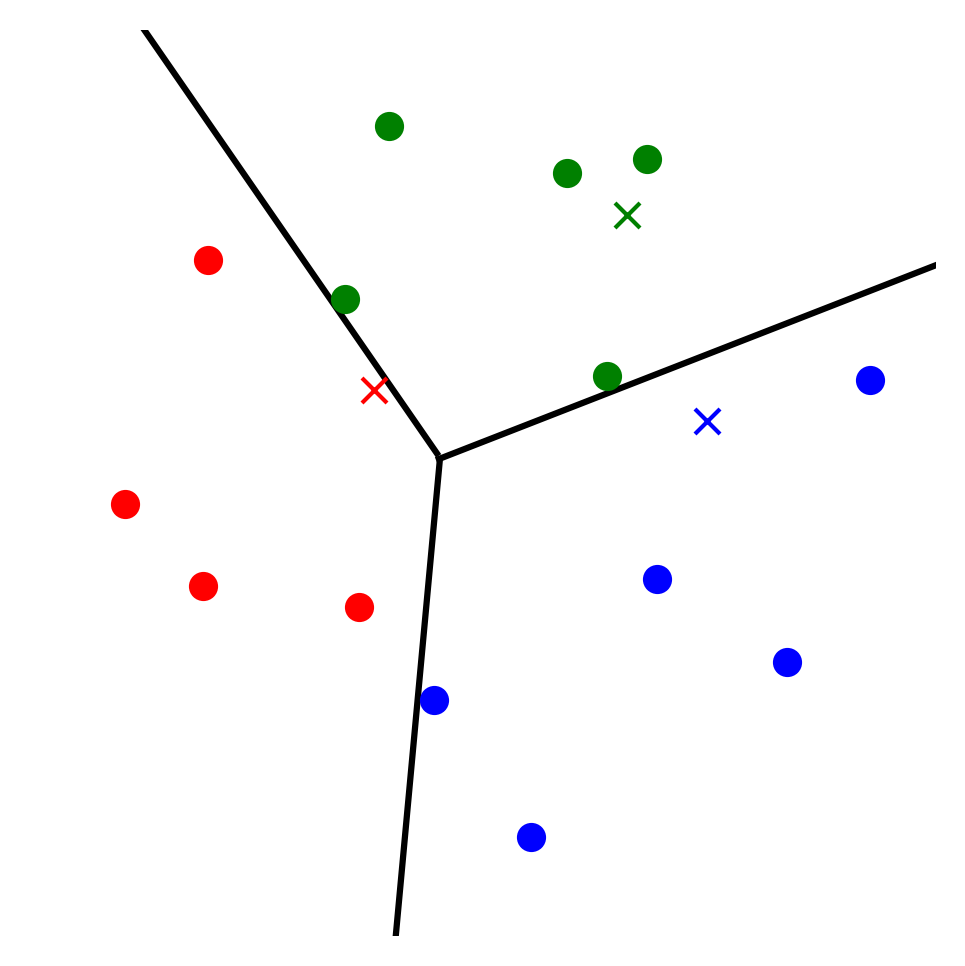} }%
\hfill
\subcaptionbox{Step of linear transition \\\hspace*{0.66cm}(sites $(1-\frac{2}{3}) s +  \frac{2}{3} t$)}{\includegraphics[width=0.4\textwidth]{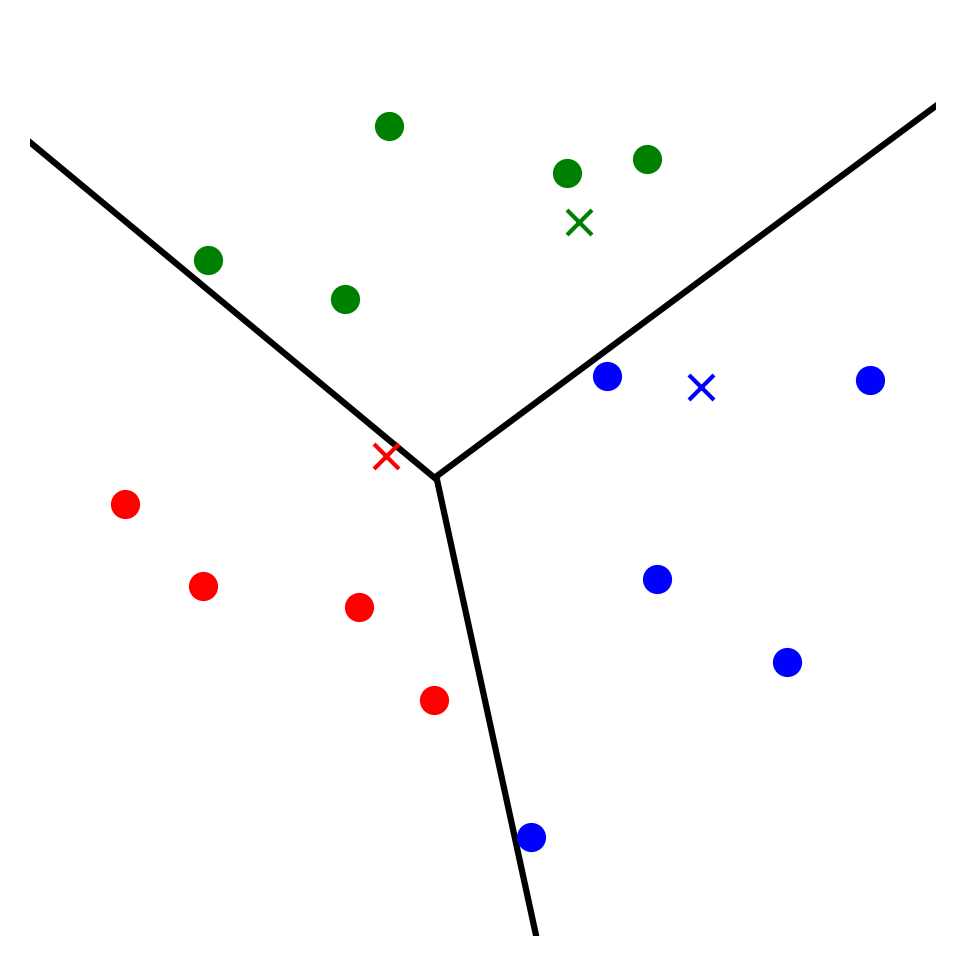}}%

\subcaptionbox{End of linear transition (sites $t$)}{\includegraphics[width=0.4\textwidth]{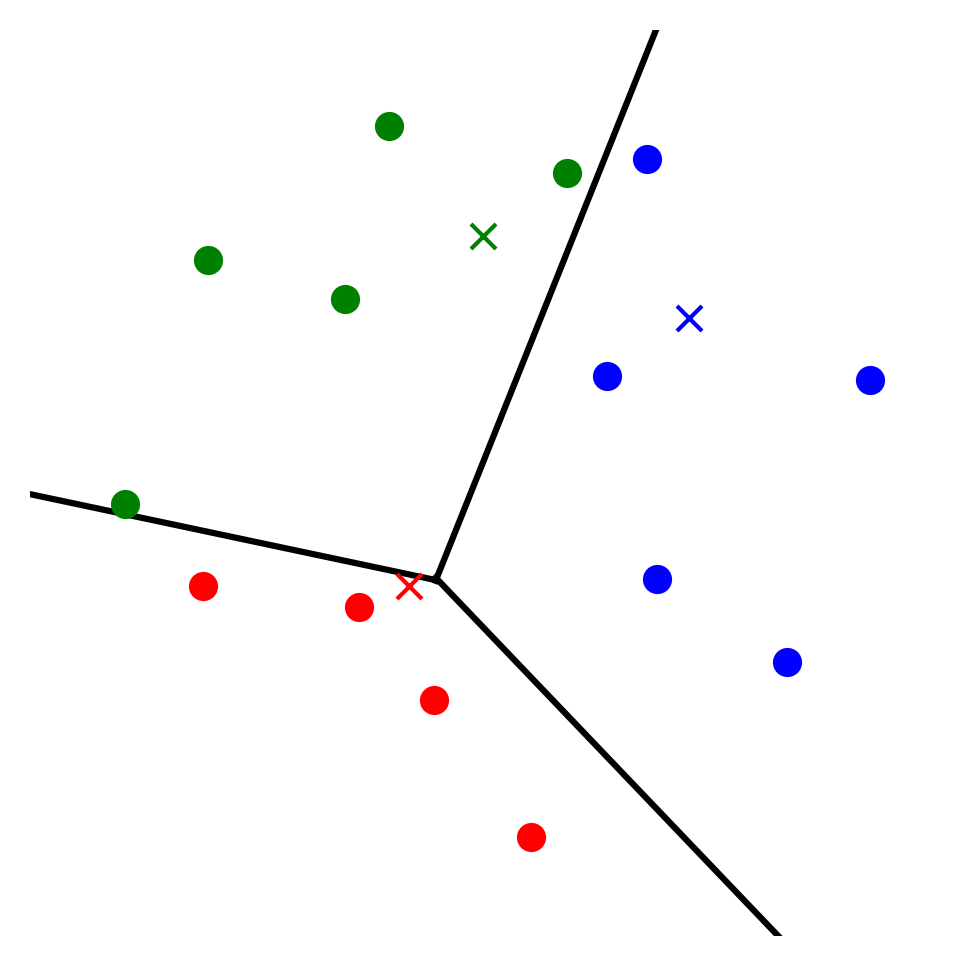}}%
\hfill
\subcaptionbox{Target LSA (sites $t$)}{\includegraphics[width=0.4\textwidth]{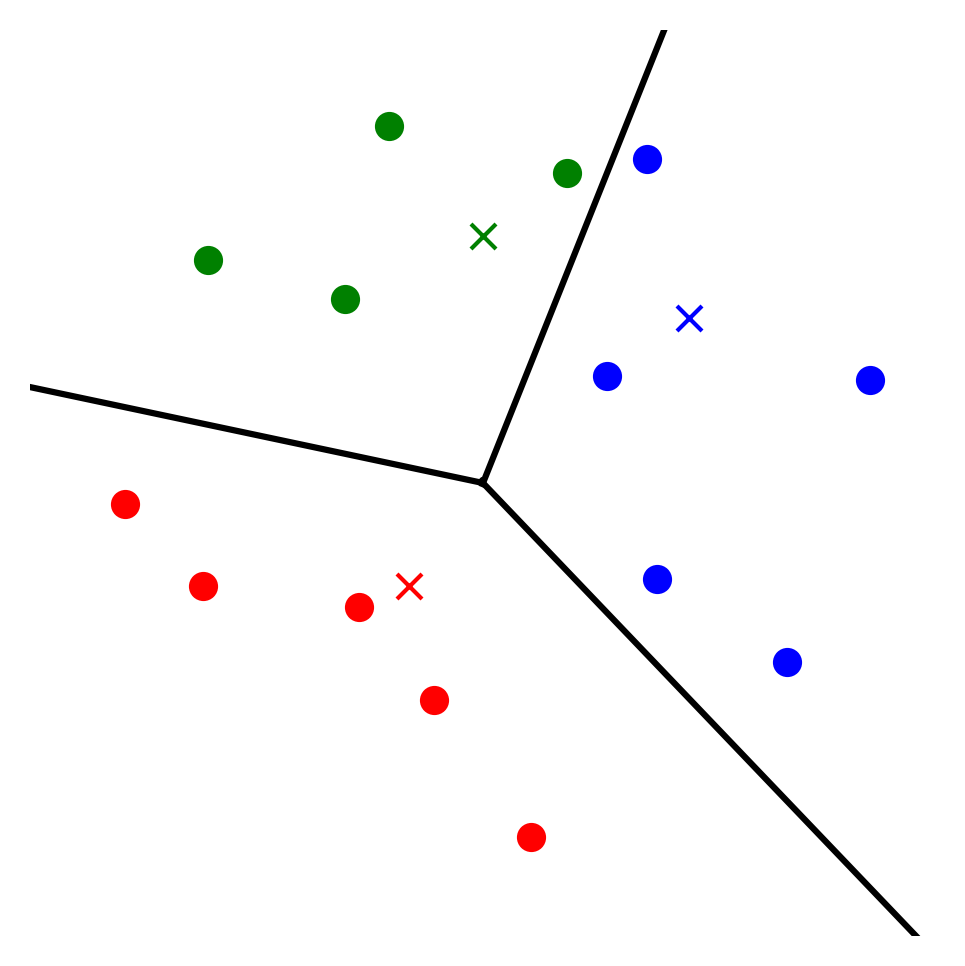}}%
\caption{A full transition from an initial LSA for sites $s$ to a target LSA for sites $t$. Consecutive clusterings differ by a single exchange. The transition may start and end with (one or more) exchanges between LSAs for the same sites. The main part of the transition follows a linear transition from sites $s$ to $t$.}
\label{fig:fulltransition}
\end{figure}


\subsection{Preserving Separability}\label{sec:upholdingseparability}

Linear programming and polyhedral theory provide several powerful tools to connect exchanges and the separability of clusterings. 
We make use of two closely connected special classes of polytopes to this end. We introduce them in Section \ref{sec:prelim:polytopes}. Section \ref{sec:prelim:separability} explains their relationship to separable clusterings; Section \ref{sec:prelim:exchanges} explains their relationship to cyclical and sequential exchanges.

\subsubsection{Bounded-Shape Partition and Transportation Polytopes}\label{sec:prelim:polytopes}

For a clustering $C$ and $i \in [k]$ and $j \in [n]$, one can introduce binary variables $y_{ij}$ that indicate whether $x_j \in C_i$ ($y_{ij} = 1$) or not ($y_{ij} = 0$).
The set of all feasible clusterings can then be described by the set of all integer vectors that satisfy the following constraints
\begin{subequations}\label{transport}
\begin{align}
&&\sum_{i = 1}^k y_{ij} &= 1 & \forall &\, j \in [n], \label{transport:partition}\\
&&\sum_{j = 1}^n y_{ij} &\geq \kappa^-_i & \forall &\, i \in [k],\label{transport:lowerbound} \\
&&\sum_{j = 1}^n y_{ij} &\leq \kappa^+_i & \forall &\, i \in [k], \label{transport:upperbound}\\
&&y_{ij} &\geq 0 & \forall &\, i \in [k], \, j \in [n]. \label{transport:nonnegative}
\end{align}
\end{subequations}

We refer to the feasible region~(\ref{transport}) as a \emph{transportation polytope}. (It shares its constraint matrix, with duplication of some of the rows, with classical transportation problems.) Constraint~(\ref{transport:partition}) ensures that each item is assigned to exactly one cluster. Constraints~(\ref{transport:lowerbound}) and~(\ref{transport:upperbound}) ensure that the clustering is feasible, i.e., each cluster $C_i$ is assigned between $\kappa^-_i$ and $\kappa^+_i$ items.
Note that the constraint matrix of~(\ref{transport}) is totally unimodular, and the right-hand side is integer.
Thus, the vertices of~(\ref{transport}) are integer, and due to Constraint~(\ref{transport:partition}) in $\{0,1\}^{n \cdot k}$. Every clustering corresponds to a vertex and vice versa. For simplicity, we identify the vertices with the encoded clusterings.

Let us provide a formal definition and notation for these polytopes, as well as for the closely connected class of \emph{partition polytopes}.

\begin{definition}[Bounded-Shape Polytopes]\label{def:boundedshape}
The feasible region~(\ref{transport}) is called the \emph{bounded-shape transportation polytope} and is denoted by $\mathcal{T}^{\pm}(X,k,\kappa^-,\kappa^+) = \mathcal{T}^{\pm}(\kappa^-,\kappa^+)$.

For a clustering $C = (C_1,\dots,C_k)$, let $w(C) \coloneqq (\sum_{x \in C_1} x,\dots,\sum_{x \in C_k} x) \in \mathbb{R}^{d \cdot k}$ be the \emph{clustering vector} of~$C$.
We call $\mathcal{P}^{\pm} (X,k,\kappa^{-},\kappa^{+}) = \mathcal{P}^{\pm}(\kappa^-,\kappa^+) \coloneqq \conv(\{w(C) \ | \ \kappa^- \leq \abs{C} \leq \kappa^+\})$ the \emph{bounded-shape partition polytope}.
\end{definition}

Informally, clustering vectors $w(C)$ are formed from $k$ vectors, one for each cluster $C_i$ summing up the locations in $\mathbb{R}^d$ of items in the cluster. Bounded-shape partition polytopes encode valuable information on the locations of items, in contrast to bounded-shape transportation polytopes which only encode item assignments to the clusters. When we write the data set into the columns of a matrix $M = (x_1,\dots,x_n)$, it is not hard to see that $\mathcal{P}^{\pm}$ is the image of $\mathcal{T}^{\pm}$ under the linear map that first transitions the vector $y \in \mathbb{R}^{n \cdot k}$ into an $(n \times k)$-matrix $Y$, multiplies this with~$M$, and rewrites the corresponding matrix $M\cdot Y$ into a column vector. Note that we may assume that the matrix $M$ has rank equal to~$d$. Otherwise, all items are contained in some affine subspace of $\mathbb{R}^d$, and we may restrict to this subspace.

Bounded-shape partition polytopes were studied in detail in the literature~\citep{BarnesHoffmanRothblum1992,HwangOnnRothblum1998,FukudaOnnRosta2003,BorgwardtHappach2019,BorgwardtViss2020}. There are two choices for $\kappa^{\pm}$ that are particularly well-understood, the \emph{single-shape} partition polytopes ~\citep{Borgwardt2010,BorgwardtHappach2019} and the \emph{all-shape} partition polytopes ~\citep{FukudaOnnRosta2003,BorgwardtHappach2019}. We define these terms.

\begin{definition}[Single-Shape and All-Shape Polytopes]\label{def:singleallshape}
For $\kappa^- = \kappa^+ = \kappa$, we call $\mathcal{P}^=(\kappa) = \mathcal{P}^{\pm}(\kappa^-,\kappa^+)$ the \emph{single-shape partition polytope} and $\mathcal{T}^{=}$ the corresponding \emph{single-shape transportation polytope}.
For $\kappa^- = (0,\dots,0)$ and $\kappa^+=(n,\dots,n)$, we call $\mathcal{P} = \mathcal{P}^{\pm}(\kappa^-,\kappa^+)$ the \emph{all-shape partition polytope} and $\mathcal{T}$ the corresponding \emph{all-shape transportation polytope}.
\end{definition}

We omit $\kappa^-$ and $\kappa^+$ and just write $\mathcal{T}^{\pm}$ or $\mathcal{P}^{\pm}$ if the cluster size bounds are clear from the context.

\subsubsection{Bounded-Shape Polytopes and Separability}\label{sec:prelim:separability}

There is a close connection between constrained LSAs and bounded-shape polytopes. As we will see, constrained LSAs are in one-to-one correspondence to vertices of single-shape partition polytopes, and all vertices (and possibly some other boundary points) of the bounded-shape partition polytopes $\mathcal{P}^{\pm}$ correspond to separable clusterings \citep{BarnesHoffmanRothblum1992,Borgwardt2010,BorgwardtHappach2019}. 

We begin by exhibiting the close connection between constrained LSAs and the single-shape transportation polytope~\citep{Borgwardt2010}.
Note that computing a constrained LSA to given sites $s_1,\dots,s_k \in \mathbb{R}^d$ is equivalent to minimizing $\sum_{i = 1}^k \sum_{j = 1}^n \norm{x_j - s_i}^2 \, y_{ij}$ over $\mathcal{T}^{\pm}$. If the cluster sizes are fixed, i.e., in the single-shape case, this is equivalent to linear optimization over the single-shape transportation polytope $\mathcal{T}^{=}$:
\begin{equation}\label{equivalence:singleshape}
\sum_{i = 1}^k \sum_{j = 1}^n \norm{x_j - s_i}^2 \, y_{ij} = \sum_{j = 1}^n \norm{x_j}^2 \underbrace{\sum_{i=1}^k y_{ij}}_{= 1 \, (\ref{transport:partition})} - 2 \sum_{i =1}^k \sum_{j = 1}^n x_j^T s_i \, y_{ij} + \sum_{i = 1}^k \norm{s_i}^2 \underbrace{\sum_{j=1}^n y_{ij}.}_{= \kappa_i \, (\ref{transport:lowerbound}), (\ref{transport:upperbound})}
\end{equation}
Thus, minimizing (\ref{equivalence:singleshape}) over all $y \in \mathcal{T}^=$ is equivalent to
\begin{equation}\label{equivalence:linearopt}
\min_{y \in \mathcal{T}^=} \underbrace{\sum_{j=1}^n \norm{x_j}^2 + \sum_{i = 1}^k \kappa_i \, \norm{s_i}^2}_{= \text{ constant}} - 2 \sum_{i =1}^k \sum_{j = 1}^n x_j^T s_i \, y_{ij} \ \Longleftrightarrow \ \max_{y \in \mathcal{T}^=} \sum_{i =1}^k \sum_{j = 1}^n x_j^T s_i \, y_{ij}.
\end{equation}
In the following, we write $c_{ij}(s) = x_j^T s_i$, so it becomes clear that (\ref{equivalence:linearopt}) maximizes the linear objective function $c(s) \in \mathbb{R}^{n \cdot k}$ over $\mathcal{T}^=$. (We will continue to use the notation $c(s)$ for such an objective function constructed from a site vector $s$ in the remainder of the paper.) Due to \begin{equation*}\sum_{i =1}^k \sum_{j = 1}^n x_j^T s_i \, y_{ij} = \sum_{i =1}^k s_i^T(\sum_{x_j \in C_i} x_j),\end{equation*} this is equivalent to a maximization of the linear objective function $s \in \mathbb{R}^{k \cdot d}$, the site vector itself, over $\mathcal{P}^=$. We obtain the following proposition.
\begin{proposition}[\cite{Borgwardt2010}]\label{prop:constrainedLSA}
A clustering $C$ is a constrained LSA to the site vector $s \in \mathbb{R}^{d \cdot k}$ if and only if it maximizes $c(s)$ over $\mathcal{T}^=$ or, equivalently, if it maximizes $s$ over $\mathcal{P}^=$. 
\end{proposition}

Under mild assumptions, the vertices of a partition polytope $\mathcal{P}^=$ are in one-to-one correspondence to separable clusterings \citep{BarnesHoffmanRothblum1992,HwangOnnRothblum1998}. In fact, they allow a separating power diagram where no items are on the boundary of cells. Again, we identify the vertices with their encoded clusterings. 

Bounded-shape polytopes are the convex hull of the union over single-shape polytopes, i.e., $\mathcal{T}^{\pm}(\kappa^-,\kappa^+) = \conv(\bigcup_{\kappa^- \leq \kappa \leq \kappa^+} \mathcal{T}^=(\kappa))$ and $\mathcal{P}^{\pm}(\kappa^-,\kappa^+) = \conv(\bigcup_{\kappa^- \leq \kappa \leq \kappa^+} \mathcal{P}^=(\kappa))$. Thus, a vertex of a bounded-shape partition polytope is also a vertex of the single-shape partition polytope corresponding to the particular shape of the encoded clustering. This leads to the following proposition.

\begin{proposition}[\cite{BarnesHoffmanRothblum1992,BorgwardtHappach2019}]\label{prop:inducedclustering}
Let $C$ be a clustering $C$ and $s \in \mathbb{R}^{d \cdot k}$ be a site vector.  If the clustering vector $w(C)$ maximizes $s$ over $\mathcal{P}^{\pm}$, then $C$ is a constrained LSA to the site vector $s$.
\end{proposition}

For a given site vector $\overline{s}$ there always exists a constrained LSA that corresponds to an optimal boundary point of $\mathcal{P}^{\pm}$ in direction of $\overline{s}$ \citep{Borgwardt2010,BriedenGritzmann2012,BorgwardtHappach2019} -- we call these {\em radial clusterings} with respect to (w.r.t.)~$\overline{s}$. However, note that Proposition~\ref{prop:inducedclustering} is not an `if and only if' relationship: a vertex of a single-shape partition polytope $\mathcal{P}^{=}$ (of a shape allowed by $\kappa^-,\kappa^+$) might not be a vertex of the bounded-shape partition polytope~$\mathcal{P}^{\pm}$. It may lie on the boundary of $\mathcal{P}^{\pm}$ or in its strict interior. The distinction between radial clusterings and general constrained LSAs will be important throughout this paper.

\subsubsection{Bounded-Shape Polytopes and Exchanges}\label{sec:prelim:exchanges}

The set of circuits or elementary vectors of bounded-shape polytopes corresponds to the desired sequential and cyclical exchanges of items \citep{bv-19a}. This implies that, under mild assumptions, two neighboring vertices of a bounded-shape partition polytope correspond to two separable clusterings that differ by just a single exchange of this type. Further, for any two separable clusterings that lie in the same face of $\mathcal{P}^{\pm}$, there exists a site vector for which the clusterings are equally good constrained LSAs. Equivalently, the clusterings allow a shared power diagram. Figure \ref{fig:SharedSepPD} showed an example.

Let us introduce some notation. We say that two clusterings are \emph{adjacent} in, e.g., $\mathcal{T}^{\pm}$ if the corresponding vertices share an edge in the transportation polytope. Further, we say $C_i$ is \emph{free} if $\kappa^-_i < \abs{C_i} < \kappa^+_i$. The following proposition summarizes the edge structure of $\mathcal{P}^{\pm}$ and $\mathcal{T}^{\pm}$.


\begin{proposition}[\cite{BorgwardtHappach2019,BorgwardtViss2020}]\label{prop:adjacentvertices}
Let $w(C),w(C')$ be two vertices of $\mathcal{P}^{\pm}$ that share an edge. Then the clustering difference graph $CDG(C,C')$ comprises a single path, a single cycle, or two parallel arcs only.

Two clusterings $C$ and $C'$ share an edge in $\mathcal{T}^{\pm}$ if and only if the clustering difference graph $CDG(C,C')$ consists of a single path in which no interior cluster is free or of a single cycle in which at most one cluster is free.
\end{proposition} 

In particular, if two clusterings are adjacent in $\mathcal{T}^{\pm}$ then their CDG consists of a single path or cycle \citep{BorgwardtViss2020}, i.e., adjacent clusterings differ by a single cyclical or sequential exchange. This is a valuable tool for our purposes: the bounded-shape transportation polytopes have an explicit algebraic description through Formulation~(\ref{transport}). This allows us to adapt classical linear programming techniques to design part of the desired transition between separable clusterings.

\section{The Overall Transition}\label{sec:transition:overview}

Given two separable clusterings, an initial clustering $C^s$ and a target clustering $C^t$ with corresponding sites $s$ and $t$, our goal is to provide an algorithm that returns a sequence of separable clusterings representing a gradual transition from $C^s$ to $C^t$. The sequence corresponds to (constrained) LSAs that appear in a linear transition of the sites from $s$ to $t$. Two consecutive clusterings in this sequence differ only by a single exchange of items.

This section is dedicated to a description of how this overarching strategy can be realized through polyhedral theory and methods of linear programming. In addition to a sequence of LSAs $$C^1,\dots,C^\eta,$$ our algorithm also returns a sequence of power diagrams $$P^1,\overline{P}^2,P^2,\overline{P}^3,\dots,P^{\eta-1},\overline{P}^{\eta},P^\eta,$$ where $P^\nu$ (uniquely) induces $C^\nu$ from the sequence and $\overline{P}^\nu$ is a shared power diagram for $C^\nu$ and $C^{\nu-1}$, for all $1 \leq \nu \leq \eta$. The computation of the power diagrams is optional, as they are not required to compute the sequence of clusterings. However, we can find them at low computational effort and they provide valuable additional information: they explicitly represent how the separation is realized. In particular, the current power diagram serves as a classifier at the current step of the transition; a new item added to the data set would be added to the cluster of the cell it falls in.

Throughout the sequence, the site vectors follow a linear transition from sites $s$ to sites $t$, i.e., $(1-\lambda) s +  \lambda t$ for increasing $\lambda$. For $\lambda = 0$ or $1$, multiple clusterings for the same sites $s$ or $t$ may appear. The site vectors of LSAs $C^\nu$ and corresponding $P^\nu$ coincide. Consecutive clusterings $C^\nu$ and $C^{\nu-1}$ are both constrained LSAs for the site vector of $\overline{P}^\nu$, i.e., their objective function values with respect to these sites are equally good. 

We begin our discussion with the computation of the sequence of clusterings, in Section \ref{sec:radialtrans}. We discuss the computation of the associated power diagrams in Section~\ref{sec:powerdiagram}.

\subsection{Computation of a Sequence of Clusterings}\label{sec:radialtrans}

First, we exhibit how to perform a {\em transition between two radial clusterings}. Recall that radial clusterings correspond to boundary points of an underlying bounded-shape partition polytope $\mathcal{P}^{\pm}$; see Proposition \ref{prop:inducedclustering}. Due to the close connection of power diagrams and linear programming over bounded-shape partition and transportation polytopes, a linear transition of radial clusterings from sites $s$ to sites $t$ is related to sensitivity analysis and so-called ranging; see, e.g., \cite{VanderbeiBook}. In a simple form of ranging, one is interested in when and how an optimal basis or vertex of a linear program changes when altering the objective function vector slightly. The {\em `breakpoint'} when an optimal (degenerate) vertex, not only an optimal basis, changes can be computed through solving a linear program (LP); see Appendix \ref{app:ranging}. This leads to an algorithm to identify {\em when} (for which values of $\lambda$) and {\em how} (which items in the data set are moved) clusterings change throughout the transition. Geometrically, we construct a walk along the boundary of $\mathcal{P}^{\pm}$.

Let us assume that initial clustering $C^s$ and target clustering $C^t$ are radial clusterings. More precisely, we assume they satisfy Proposition \ref{prop:inducedclustering} with respect to the bounded-shape partition polytope $\mathcal{P}^{\pm}=\mathcal{P}^{\pm}(\theta^{-},\theta^{+})$ using cluster size bounds $\theta^-_i \leq \kappa^-_i = \min\set{\abs{C^s_i},\abs{C^t_i}}$ and $\theta^+_i\geq \kappa^+_i = \max\set{\abs{C^s_i},\abs{C^t_i}}$ for all $i \in [k]$, and we assume the clustering vectors $w(C^s)$ and $w(C^t)$ are on the boundary of $\mathcal{P}^{\pm}$ in direction of $s \in \mathbb{R}^{d \cdot k}$ and $t \in \mathbb{R}^{d \cdot k}$, respectively. Recall that not all constrained LSAs satisfy this property even if they are feasible w.r.t.~the given $\theta^\pm$.

The cluster sizes $\theta^\pm$ are chosen such that the shapes of $C^s$ and $C^t$ lie between the given bounds. They represent bounds on the shapes of any clusterings that can appear throughout the transition. For a simple notation, we describe our approach assuming $\theta^\pm=\kappa^\pm$ in the following. However, we would like to stress that a wider range of cluster sizes could be chosen and that this is not a technical restriction. In fact, Figure \ref{fig:fulltransition} showed an example in which the shapes of the intermediate clusterings are from a wider range. The choice of bounds $\kappa^\pm$ comes with the additional benefit of keeping the cluster sizes in the intermediate clusterings as similar as possible to $C^s$ and $C^t$; it is the smallest range of cluster sizes that allows for a transition. In Appendix \ref{app:radial}, we provide another reason why setting $\theta^\pm=\kappa^\pm$ is a good choice: the wider the range of feasible cluster sizes, the more `special' radial clusterings become compared to general constrained LSAs. 

When moving linearly from the site vector $s$ to $t$, we obtain a sequence of site vectors of the form $(1-\lambda)s + \lambda t \in \mathbb{R}^{d \cdot k}$ with $\lambda \in [0;1]$. For every $\lambda$, we get a clustering vector $w(C^\lambda)$ that maximizes the objective vector $ (1-\lambda)s + \lambda t$ over $\mathcal{P}^{\pm}$, where $C^0 = C^s$ and $C^1 = C^t$.
That is, the clustering $C^\lambda$ is separable and is induced by a power diagram with site vector $ (1-\lambda)s + \lambda t$.

Recall that an objective vector is maximized at a boundary point $x$ of a polytope $P$ if and only if it is contained in the normal cone $N_{P}(x) := \set{c \in \mathbb{R}^d \, | \, c^T x \geq c^T x' \ \forall \, x' \in P}$.
Let $C$ be a clustering that is induced by a power diagram with site vector $\overline s \in \mathbb{R}^{d \cdot k}$ and let $C'$ be any other clustering. Let $i(j),i'(j) \in [k]$ be the indices such that $x_j \in C_{i(j)}$ and $x_j \in C'_{i'(j)}$ for all $j \in [n]$.
Let $y,y' \in \set{0,1}^{n \cdot k}$ be the feasible solutions for the bounded-shape transportation polytope $\mathcal{T}^{\pm}$ that correspond to $C$ and $C'$, respectively.
We have the following equivalence:
\begin{equation*}
\begin{split}
\overline s \in N_{\mathcal{P}^{\pm}}(w(C)) &\Longleftrightarrow \ \overline s^T w(C) \geq \overline s^T w(C') \ \Longleftrightarrow \ \sum_{i=1}^k \overline s_i^T \bigg(\sum_{\substack{j\in [n] \\ i = i(j)}} x_j \bigg) \geq \sum_{i=1}^k \overline s_i^T \bigg(\sum_{\substack{j \in [n] \\ i=i'(j)}} x_j \bigg) \\
&\Longleftrightarrow \ \sum_{j=1}^n x_j^T \overline s_{i(j)} \geq \sum_{j=1}^n x_j^T \overline s_{i'(j)} \ \Longleftrightarrow \ \sum_{i=1}^k \sum_{j=1}^n c(\overline s)_{ij} y_{ij} \geq \sum_{i=1}^k \sum_{j=1}^n c(\overline s)_{ij} y'_{ij} \\
& \Longleftrightarrow \ c(\overline s)^T y \geq c(\overline s)^T y' \ \Longleftrightarrow \ c(\overline s) \in N_{\mathcal{T}^{\pm}}(y).
\end{split}
\end{equation*}

Hence, the clustering vectors $w(C^\lambda)$ are radial, i.e., on the boundary of $\mathcal{P}^{\pm}$, if and only if their 0/1 vectors in $\mathcal{T}^{\pm}$ are optimal vertices in direction of $c((1-\lambda)s + \lambda t)$. Thus, a linear transition from $s$ to $t$ w.r.t.~$\mathcal{P}^{\pm}$ is equivalent to a linear transition from $c(s)$ to $c(t)$ w.r.t.~$\mathcal{T}^{\pm}$. Instead of moving along the boundary of the bounded-shape partition polytope, we can move along the boundary of the bounded-shape transportation polytope. Recall that feasible clusterings are in one-to-one correspondence to the vertices of $\mathcal{T}^{\pm}$ and adjacent vertices differ by a single cyclical or sequential exchange (Proposition~\ref{prop:adjacentvertices}). This implies that we can perform the walk along the boundary of $\mathcal{T}^{\pm}$ as an {\em edge walk}. Further, note we have an explicit representation of $\mathcal{T}^{\pm}$ (Formulation (\ref{transport})), so we have all the necessary information for ranging for a current clustering; see Appendix \ref{app:ranging}. This allows for the design of an iterative scheme, similar to parametric optimization, that finds the values of $\lambda$ for which the current clustering changes during the transition $(1-\lambda)s + \lambda t$ (for increasing $\lambda$), as well as to identify the exchange of items by which two consecutive clusterings differ. 

These steps are summarized in \textsc{AlgoRadToRad}, see Algorithm~\ref{alg:rad-to-rad}, which is devised and discussed in Section~\ref{sec:rad-to-rad}. Geometrically, the algorithm returns a walk along the boundary of $\mathcal{P}^{\pm}$ (and an edge walk in $\mathcal{T}^{\pm}$) from an initial radial clustering $C^s_\text{rad}$ to a target radial clustering $C^t_\text{rad}$.


However, recall that not all LSAs are radial clusterings. Suppose that the clustering vector of the initial clustering $C^s$ or target clustering $C^t$ are not radial, i.e., they are not on the boundary of $\mathcal{P}^{\pm}$. For example, if $C^s$ is not radial w.r.t.~$s$, then the clustering $C^s$ is a constrained LSA, i.e., a vertex of the corresponding single-shape partition polytope $\mathcal{P}^=$, but the 0/1 vector of $C^s$ is not an optimum vertex of $\mathcal{T}^{\pm}$ in direction of $c(s)$, but only optimal over $\mathcal{T}^=$.

For such situations, we design a second algorithm for a {\em transition of a constrained LSA to a corresponding radial clustering} with respect to the same sites $s$ or $t$. As in the previous algorithm, each step of this transition is the application of a single exchange of items. Each intermediate clustering is a constrained LSA for the sites $s$ or $t$, respectively; only the number of items in the clusters changes. Because of this property, we call this a {\em fixed-size transition}. 

The computation of such a transition is described in \textsc{AlgoLSAtoRad}, see Algorithm~\ref{alg:init-to-rad}, which is devised and discussed in Section~\ref{sec:init-to-rad}. In particular, we show that all intermediate clusterings on the transition from $C^s$ to $C^s_\text{rad}$ are constrained LSAs w.r.t.~$s$ and can be constructed such that two consecutive clusterings differ by only a single sequential exchange. The same holds for the transition between $C^t$ and $C^t_\text{rad}$. Note that the transition from $C^t_\text{rad}$ to $C^t$ can be computed by applying \textsc{AlgoLSAtoRad} and reversing the order of the returned sequence of clusterings. We obtain a fixed-site transition for the first and final part of the overall transition.

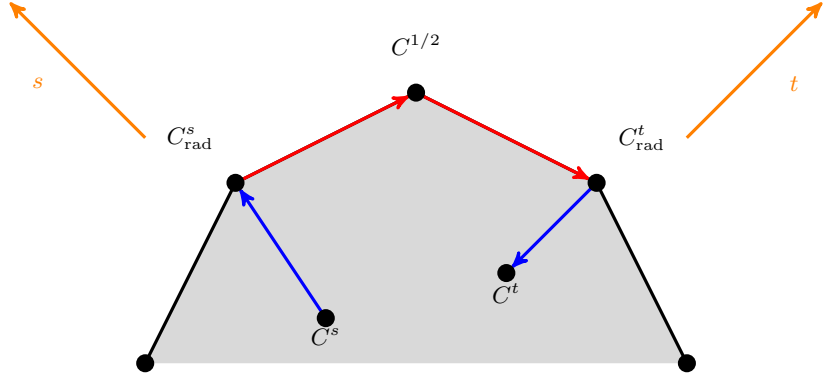
\begin{figure}
    \centering
    \begin{tikzpicture}[very thick,scale=1.2]
        \coordinate (v1) at (0,0);
        \coordinate (v2) at (1,2);
        \coordinate (v3) at (3,3);
        \coordinate (v4) at (5,2);
        \coordinate (v5) at (6,0);
        \coordinate (s) at (2,0.5);
        \coordinate (t) at (4,1);

        \fill[gray!30] (v1) -- (v2) -- (v3) -- (v4) -- (v5) -- cycle;
        \draw (v1) -- (v2) -- (v3) -- (v4) -- (v5);
        
        \draw[->,blue,shorten >=3pt,>=stealth'] (s) -- (v2);
        \draw[->,blue,shorten >=3pt,>=stealth'] (v4) -- (t);
        \draw[->,red,shorten >=3pt,>=stealth'] (v2) -- (v3);
        \draw[->,red,shorten >=3pt,>=stealth'] (v3) -- (v4);
        
        \foreach \i in {1,2,3,4,5}{\fill (v\i) circle (0.1);}
        \fill (s) circle (0.1) node[below] {$C^s$};
        \fill (t) circle (0.1) node[below] {$C^t$};
        \node[above left of=v2,xshift=3pt,yshift=-3pt] {$C^s_\text{rad}$};
        \node[above right of=v4,xshift=-3pt,yshift=-3pt] {$C^t_\text{rad}$};
        \node[above of=v3,yshift=-10pt]{$C^{1/2}$};
        
        \draw[->,orange,>=stealth'] (0,2.5) -- (-1.5,4) node[midway,xshift=-15pt,yshift=-5pt]{$s$};
        \draw[->,orange,>=stealth'] (6,2.5) -- (7.5,4) node[midway,xshift=15pt,yshift=-5pt]{$t$};
        
    \end{tikzpicture}
    \caption{Sketch of a transition from $C^s$ to $C^t$ over $\mathcal{P}^{\pm}$. Clustering $C^s$ is first transitioned into a radial clustering $C^s_\text{rad}$ for the same sites; then $C^s_\text{rad}$ is transitioned into $C^t_\text{rad}$, the radial clustering corresponding to $C^t$, through a walk along the boundary of the polytope; finally $C^t_\text{rad}$ is transitioned into $C^t$. The current site vectors stay fixed in the blue part of the transition, and change linearly from $s$ to $t$ in the red part.}
    \label{fig:lineartransition}
\end{figure}

A combination of these two algorithms allows for the computation of a transition of any constrained LSA to any other through a sequence of cyclical and sequential exchanges of items, while retaining separable clusterings throughout. A conceptual sketch of the overall transition from $C^s$ to $C^t$ is shown in Figure~\ref{fig:lineartransition}. The shaded area represents an underlying bounded-shape partition polytope $\mathcal{P}^\pm$, as well as two site vectors $s$ and $t$. Clusterings $C^s$ and $C^t$ are constrained LSAs for these site vectors, but not radial clusterings w.r.t.~$\mathcal{P}^\pm$. The transitions from the initial clustering $C^s$ to its radial clustering $C^s_\text{rad}$ and from the radial clustering $C^t_\text{rad}$ to the target clustering $C^t$ are depicted in blue. They begin or stop at the radial `counterparts' of $C^s$ and $C^t$, respectively. The transition between the two radial clusterings is depicted in red in Figure~\ref{fig:lineartransition}. The intermediate clustering named $C^{1/2}$ is radial w.r.t.~$\mathcal{P}^\pm$ for the site vector $\frac{1}{2}s+\frac{1}{2}t$.

Algorithm~\ref{alg:overalltransition} summarizes our approach. We formally state correctness of Algorithm~\ref{alg:overalltransition}, and the many favorable properties of the constructed walk, in the following theorem. 

\IncMargin{1em}
\begin{algorithm}[t]

	\Indm
	\KwIn{Initial and target constrained LSAs $C^s$ and $C^t$ w.r.t.~site vectors $s \in \mathbb{R}^{d \cdot k}$ and $t \in \mathbb{R}^{d \cdot k}$, respectively.}
	\KwOut{Sequence of constrained LSAs and sequence of corresponding power diagrams that satisfy the properties of Theorem \ref{thm:overalltransition}}
	\Indp
	
	\medskip
	
	Set $\kappa^-_i \leftarrow \min\set{\abs{C^s_i},\abs{C^t_i}}$ and $\kappa^+_i \leftarrow \max\set{\abs{C^s_i},\abs{C^t_i}}$ for all $i \in [k]$\;
	Call \textsc{AlgoLSAtoRad}($C^s,s,\kappa^-,\kappa^+$) and let $(C^{s,0},C^{s,1},\dots,C^{s,p})$ be the returned sequence of clusterings where $C^{s,0} = C^s$ and let $(P^{s,0},\overline{P}^{s,1},\dots,\overline{P}^{s,p},P^{s,p})$ be the returned sequence of power diagrams; set $C^s_\text{rad} \leftarrow C^{s,p}$\;
	
	Call \textsc{AlgoLSAtoRad}($C^t,t,\kappa^-,\kappa^+$) and let $(C^{t,0},C^{t,1},\dots,C^{t,q})$ be the returned sequence of clusterings where $C^{t,0} = C^t$ and let $(P^{t,0},\overline{P}^{t,1},\dots,\overline{P}^{t,q},P^{t,q})$ be the returned sequence of power diagrams; set $C^t_\text{rad} \leftarrow C^{t,q}$\;
	
	Call \textsc{AlgoRadToRad}($C^s_\text{rad},C^t_\text{rad},s,t,\kappa^-,\kappa^+$) and let $(C^{\lambda_0},C^{\lambda_1},\dots,C^{\lambda_m})$ be the returned sequence of clusterings where $C^{\lambda_0} = C^s_\text{rad}$ and $C^{\lambda_m} = C^t_\text{rad}$ and let $(\overline{P}^{\lambda_1},P^{\lambda_1},\dots,\overline{P}^{\lambda_{m-1}},P^{\lambda_{m-1}},\overline{P}^{\lambda_m})$ be the returned sequence of power diagrams\;
	
    \Return sequence of clusterings\vspace*{-0.25cm} $$(C^{s,0},C^{s,1},\dots,C^{s,p},C^{\lambda_1},\dots,C^{\lambda_m},C^{t,q-1},\dots,C^{t,1},C^{t,0})\vspace*{-0.25cm}$$ \hspace*{1.18cm}and sequence of power diagrams\vspace*{-0.25cm} $$(P^{s,0},\overline{P}^{s,1},\dots,\overline{P}^{s,p},P^{s,p},\overline{P}^{\lambda_1},\dots,P^{\lambda_{m-1}},\overline{P}^{\lambda_m},P^{t,q},\overline{P}^{t,q},P^{t,q-1},\dots,\overline{P}^{t,1},P^{t,0})\;$$
   
\caption{Linear transition from initial LSA~$C^s$ to target LSA~$C^t$.}
\label{alg:overalltransition}
\end{algorithm}
\DecMargin{1em}

\begin{theorem}\label{thm:overalltransition}
Let $C^s$ and $C^t$ be two constrained LSAs with site vectors $s$ and $t$. Algorithm~\ref{alg:overalltransition} returns a sequence of clusterings $$(C^{s,0},C^{s,1},\dots,C^{s,p},C^{\lambda_1},\dots,C^{\lambda_m},C^{t,q-1},\dots,C^{t,1},C^{t,0})$$ and a sequence of power diagrams $$(P^{s,0},\overline{P}^{s,1},\dots,\overline{P}^{s,p},P^{s,p},\overline{P}^{\lambda_1},\dots,P^{\lambda_{m-1}},\overline{P}^{\lambda_m},P^{t,q},\overline{P}^{t,q},P^{t,q-1},\dots,\overline{P}^{t,1},P^{t,0})$$
that satisfy the following properties:

\begin{enumerate}
    \item $C^{s,0}=C^s$, $C^{t,0}=C^t$
    \item all clusterings are constrained LSAs (and thus separable) 
    \item all clusterings $C$ have cluster sizes $|C_i|$ satisfying $\kappa_i^-:=\min \{|C_i^s|, |C^t_i|\} \leq |C_i| \leq \kappa_i^+:= \max \{|C_i^s|, |C^t_i|\}$  for all $i\leq k$ 
    \item consecutive clusterings differ by a single cyclical or sequential exchange of items
    \item The power diagrams in the sequence satisfy:\begin{itemize}\item $P^{s,i}$ is a separating power diagram for sites $s$ for $C^{s,i}$ \item $P^{\lambda_i}$ is a separating power diagram for $C^{\lambda_i}$ \item $P^{t,i}$ is a separating power diagram for sites $t$ for $C^{t,i}$\end{itemize} 
    \item The clusterings in the sequence satisfy:\begin{itemize}\item for consecutive clusterings $C^{s,i-1}$, $C^{s,i}$, $\overline{P}^{i}$ is a shared power diagram for sites $s$\item for consecutive clusterings $C^{\lambda_i}, C^{\lambda_{i+1}}$,  $\overline{P}^{\lambda_i}$ is a shared power diagram for sites $\small (1{-}\lambda_i) s{+}\lambda_i t $\item for consecutive clusterings $C^{t,i}$, $C^{t,i-1}$, $\overline{P}^{i}$ is a shared power diagram for sites $t$\end{itemize}
    \item $C^s_{\text{rad}}=C^{s,p}=C^{\lambda_0}, C^{\lambda_1},\dots,C^{\lambda_m}=C^{t,q}=C^t_{\text{rad}}$ are radial clusterings for sites $(1-\lambda_i) s +\lambda_i t $ for all $i\geq 0$ w.r.t.~$\mathcal{P}^{\pm}(\kappa^-,\kappa^+)$ and $\;\;\mathcal{T}^{\pm}(\kappa^-,\kappa^+)$ (with $\kappa^\pm$ as defined in 3.) 
    \item $C^{s,0},C^{s,1},\dots,C^{s,p}$ are constrained LSAs for sites $s$
    \item $C^{t,q-1},\dots,C^{t,1},C^{t,0}$ are constrained LSAs for sites $t$
    \item the shapes $|C^{s,i}|$ of $C^{s,0},C^{s,1},\dots,C^{s,p}$ are all distinct; the number of clusterings in this part of the sequence is bounded by the number of shapes
    \item the shapes $|C^{t,q}|$ of $C^{t,q},C^{t,q-1},\dots,C^{t,0}$ are all distinct; the number of clusterings in this part of the sequence is bounded by the number of shapes.
    \end{enumerate}
\end{theorem}

\begin{proof}{Proof.}
 Note that Algorithm~\ref{alg:overalltransition} essentially consists of two calls to \textsc{AlgoLSAtoRad} (lines $2$ and $3$) and a call to \textsc{AlgoRadToRad} (line $4$). The two calls to \textsc{AlgoLSAtoRad} provide a transition of $C^s$ into radial $C^s_\text{rad}$ and $C^t$ into radial $C^t_\text{rad}$. The call to \textsc{AlgoRadToRad} provides a transition of $C^s_\text{rad}$ to $C^t_\text{rad}$. Along with these sequences of clusterings, a corresponding sequence of power diagrams is computed. The returned walk is a concatenation of the sequence of clusterings from $C^s$ to $C^s_\text{rad}$ (line $2$), from $C^s_\text{rad}$ to $C^t_\text{rad}$ (line 4), and from $C^t_\text{rad}$ to $C^t$ (line 3). The latter is a simple reversal of the transition from $C^t$ into radial $C^t_\text{rad}$ computed in line $3$. The order for the calls to the algorithms, in particular the call to \textsc{AlgoLSAtoRad} in line $3$ before the call to \textsc{AlgoRadToRad} in line $4$, is based on the fact that \textsc{AlgoRadToRad} requires $C^t_\text{rad}$ as input, and $C^t_\text{rad}$ is found as a part of the earlier run of \textsc{AlgoLSAtoRad}. 

The claimed properties of the constructed sequence of clusterings are now a direct consequence of correctness, and precisely these properties, of the three parts of the walk. We prove correctness of \textsc{AlgoLSAtoRad} in Theorem \ref{thm:init-to-rad} and correctness of \textsc{AlgoRadToRad} in Theorem \ref{thm:rad-to-rad}. 

Claim 1 (in this theorem) corresponds to property 1 in Theorem \ref{thm:init-to-rad}. Claims 2 to 6 follow from properties 2 to 6, respectively, in Theorems \ref{thm:init-to-rad} and \ref{thm:rad-to-rad}. Claim 7 follows from property 1 in Theorem \ref{thm:init-to-rad} and property 2 in Theorem \ref{thm:init-to-rad}. Claims 8 and 9 correspond to property 2 in Theorem \ref{thm:init-to-rad}, and claims 10 and 11 correspond to property 7 in Theorem \ref{thm:init-to-rad}.
\hfill
\qed 
\end{proof}

Some of the steps in Algorithm~\ref{alg:overalltransition} can be simplified if the cluster sizes of the initial and target clustering coincide. Then $\kappa^- = \kappa^+$ and Proposition~\ref{prop:constrainedLSA} yields that a clustering is an LSA w.r.t.~a site vector $a \in \mathbb{R}^{d \cdot k}$ if and only if it is radial w.r.t.~$a$. Thus $C^s = C^s_\text{rad}$ and $C^t = C^t_\text{rad}$, and one can omit \textsc{AlgoLSAtoRad}.

\subsection{Computation of Separating Power Diagrams}\label{sec:powerdiagram}

In addition to the construction of a sequence of clusterings, we construct a sequence of corresponding separating power diagrams. Let $\nu$ be an iterator for the sequence of clusterings. The power diagrams are of two types: power diagrams $P^\nu=P^{s,i},P^{\lambda_i},P^{t,i}$ are supposed to be `good' power diagrams inducing clustering $C^\nu=C^{s,i},C^{\lambda_i},C^{t,i}$ in the sequence (and no other clusterings from the sequence); power diagrams $\overline{P}^\nu=\overline{P}^i,\overline{P}^{\lambda_i}$ are shared for two consecutive clusterings $C^{\nu - 1}$ and $C^\nu$, inducing both $C^{s,i-1}, C^{s,i}$ or  $C^{t,i}$, $C^{t,i-1}$, or $C^{\lambda_i}, C^{\lambda_{i+1}}$, depending on the part of the sequence. In this section, we discuss the computation of these power diagrams.

To this end, recall from Proposition~\ref{prop:inducedclustering} that for every clustering $C$ with clustering vector $w(C)$ on the boundary of $\mathcal{P}^{\pm}$ or $\mathcal{P}^=$, and every vector $\overline{s} \in \mathbb{R}^{d \cdot k}$ that is contained in the normal cone at $w(C)$, there exists a power diagram for site vector $\overline{s}$ that induces the clustering $C$.
The scalars $\gamma_1,\dots,\gamma_k$ in Definition~\ref{def:powerdiagram} only depend on $X$ and $\overline{s}$ and can be constructed for any $\overline{s}$ in the normal cone~\citep{BarnesHoffmanRothblum1992,Borgwardt2010,BriedenGritzmann2012}.

Let us first discuss the construction of a good power diagram $P^\nu$ that induces clustering $C^\nu$ in the sequence: we want a good site vector, and good positions of the separating hyperplanes. First, note that we are given a site vector in the normal cone for every clustering $C^\nu$. For the clusterings $C^{s,0},\dots,C^{s,p}$ and $C^{t,q},\dots,C^{t,0}$ (blue arcs in Figure~\ref{fig:lineartransition}) that are returned by \textsc{AlgoLSAtoRad}, we know that $s$ and $t$ are contained in the normal cone of the clustering vector in the respective single-shape partition polytopes. For the clusterings $C^{\lambda_r}$ ($1 \leq r \leq m$) along the linear transition $(1-\lambda) s + \lambda t $ (red arcs in Figure~\ref{fig:lineartransition}), the value $\lambda_r \in [0;1]$ is chosen such that $s^{\lambda_r} =  (1 - \lambda_r) s + \lambda_r t$ is contained in the intersection of the normal cones of $w(C^{\lambda_{r-1}})$ and $w(C^{\lambda_r})$, as we will see in \textsc{AlgoRadToRad} and Section~\ref{sec:rad-to-rad}.
Hence, a natural choice of a site vector that is in the normal cone at $w(C^{\lambda_r})$, but not in the normal cone of the previous or next clustering vector, is a convex combination of $s$ and $t$ in the form $\frac{1}{2}(s^{\lambda_r} + s^{\lambda_{r+1}})$ for $1 \leq r < m$. Note that this choice of site vector for $C^{\lambda_r}$ implies that the constructed power diagram uniquely induces $C^{\lambda_r}$ from the sequence. (The clustering vectors of the first and last clustering of this sequence $C^{\lambda_0}= C^s_\text{rad}$ and $C^{\lambda_m} = C^t_\text{rad}$ contain $s$ and $t$ in their normal cone by construction, respectively.) Summing up, there is a natural choice for the site vector for each power diagram $P^\nu$. These follow a linear transition from $s$ to $t$. Second, for given sites, we also want good positions of the separating hyperplanes in the underlying space. 
A common approach is to maximize the so-called \emph{margin}, i.e., maximize the minimum Euclidean distance of any item to the boundary of its cell; see, e.g., \cite{bm-92,Borgwardt2015}. 

It remains to address the actual computation of a (maximum-margin) power diagram for a given site vector. Given a clustering $C$ and a site vector $\overline{s} = (\overline{s}_1,\dots,\overline{s}_k) \in N_{\mathcal{P}^=(\abs{C})}(w(C))$, we can maximize the margin $\varepsilon$ of a power diagram with site vector $\overline{s}$ by solving the following LP \citep{bm-92,Borgwardt2015}:
\begin{subequations}\label{LP:maximizemargin}
\begin{align}
& \max & \varepsilon & \span \span \label{LP:maximizemargin:objective}\\
& \text{s.t.} & \left(\frac{\overline{s}_\ell - \overline{s}_i}{\|\overline{s}_\ell - \overline{s}_i\|}\right)^T x_j + \varepsilon &\leq \frac{\gamma_\ell - \gamma_i}{\|\overline{s}_\ell - \overline{s}_i\|} & \forall & \, i \in [k], \forall \, \ell \in [k] \setminus \{i\}, \forall \, x_j \in C_i, \label{LP:maximizemargin:feasible} \\
\span & \gamma_1 & = 0,\\
\span & \gamma_i &\in \mathbb{R} & \forall & \, i \in [k], \label{LP:maximizemargin:gamma}\\
\span & \varepsilon &\geq 0. \label{LP:maximizemargin:epsilon}
\end{align}
\end{subequations}

A notable feature of this LP is the constraint $\gamma_1 =0$; without it, whenever the system has a feasible solution, it would also contain a one-dimensional lineality space. LP~(\ref{LP:maximizemargin}) is guaranteed to have a feasible solution if and only if there exists a power diagram with site vector $\overline{s}$ that induces~$C$. If $\varepsilon > 0$ no item is on the boundary of the cell of the resulting power diagram. In fact, because of the scaling by $\|\overline{s}_\ell - \overline{s}_i\|^{-1}$ in~(\ref{LP:maximizemargin:feasible}), $\epsilon$ corresponds to the minimal {\em Euclidean} distance of an item to the boundary of its cell. We can simplify the LP if we first compute $\max_{x_j \in C_i} (\overline{s}_\ell - \overline{s}_i)^T x_j$ for all $i \in [k] $ and $\ell \in [k] \setminus \{i\}$, and use this term in the constraints~(\ref{LP:maximizemargin:feasible}). 
The resulting LP takes the smaller and simpler form

\begin{subequations}\label{LP:computeweights}
\begin{align}
& \max &  \varepsilon & \span \span \label{LP:computeweights:objective}\\
& \text{s.t.} & \left(\frac{\max_{x_j \in C_i} (\overline{s}_\ell - \overline{s}_i)}{\|\overline{s}_\ell - \overline{s}_i\|}\right)^T x_j + \varepsilon &\leq \frac{\gamma_\ell - \gamma_i}{\|\overline{s}_\ell - \overline{s}_i\|} & \forall & \, i \in [k], \forall \, \ell \in [k] \setminus \{i\}, \label{LP:computeweights:feasible}\\
\span & \gamma_1 & = 0, \label{LP:computeweights:normalize}\\
\span & \gamma_i &\in \mathbb{R} & \forall & \, i \in [k], \label{LP:computeweights:weights}\\
\span & \varepsilon &\geq 0. \label{LP:computeweights:epsilon}
\end{align}
\end{subequations}

Note that LP~(\ref{LP:computeweights}) has only $k$ variables and less than $k^2$ main constraints, and setup only requires processing each item $k-1$ times. Compared to Formulation~(\ref{transport}), which is an integral part of the two algorithms for the computation of the sequence of clusterings, it is trivial to solve. 

Finally, for the construction of the shared power diagram $\overline{P}^\nu$ that induces both $C^{\nu - 1}$ and $C^\nu$, suppose we are given a site vector $\overline{s} \in \mathbb{R}^{d \cdot k}$ that is contained in the intersection of the normal cones of $w(C^{\nu-1})$ and $w(C^{\nu})$. For the clusterings $C^{\lambda_r}, C^{\lambda_{r-1}}$ ($1 \leq r \leq m$) along the linear transition $(1-\lambda) s + \lambda t $ (red arcs in Figure~\ref{fig:lineartransition}), we can choose the site vector $\overline{s} = s^{\lambda_{r}} = (1 - \lambda_{r}) s + \lambda_{r} t$ for $P^{\lambda_r}$, see \textsc{AlgoRadToRad}. For two consecutive clusterings in $\set{C^{s,0},\dots,C^{s,p}}$ and $\set{C^{t,q},\dots,C^{t,0}}$ (blue arcs in Figure~\ref{fig:lineartransition}), $\overline{s} \in \set{s,t}$ allows for the computation of a shared power diagram, too. The somewhat technical proof of this claim is part of the proof of Theorem \ref{thm:init-to-rad} in Section \ref{sec:init-to-rad}.

In both situations, we obtain a shared power diagram that induces $C^{\nu-1}$ and $C^{\nu}$ by adding the constraints~(\ref{LP:computeweights:feasible}) for $C^{\nu-1}$ to LP~(\ref{LP:computeweights}) for $C^\nu$. The items by which $C^{\nu-1}$ and $C^{\nu}$ differ lie on the boundary of the respective cells; see~\cite{Borgwardt2010,BorgwardtHappach2019}. Thus, the optimal objective value of LP~(\ref{LP:computeweights}) is equal to zero and computing a power diagram that induces both clusterings simultaneously can be done by finding a feasible solution $\gamma = (\gamma_1,\dots,\gamma_k)$ to the following set of linear constraints that does {\em not} contain a scaling by $\|\overline{s}_\ell - \overline{s}_i\|^{-1}$:
\begin{subequations}\label{LP:computeweights:boundarypowerdiagram}
\begin{align}
&& \max_{x_j \in C^\nu_i} (\overline{s}_\ell - \overline{s}_i)^T x_j  &\leq \gamma_\ell - \gamma_i & \forall & \, i \in [k], \forall \, \ell \in [k] \setminus \{i\}, \label{LP:computeweights:boundarypowerdiagram:feasible}\\
&& \max_{x_j \in C^{\nu-1}_i} (\overline{s}_\ell - \overline{s}_i)^T x_j  &\leq \gamma_\ell - \gamma_i & \forall & \, i \in [k], \forall \, \ell \in [k] \setminus \{i\}, \label{LP:computeweights:boundarypowerdiagram:feasible2}\\
\span & \gamma_1 & = 0, \label{LP:computeweights:boundarypowerdiagram:normalize} \\
\span & \gamma_i &\in \mathbb{R} & \forall & \, i \in [k]. \label{LP:computeweights:boundarypowerdiagram:weights}
\end{align}
\end{subequations}
This system can be considered an LP with trivial objective function $\max 0^T\gamma$ or solved through a Fourier-Motzkin elimination. 

\section{A Fixed-Site Transition between LSAs and Radial Clusterings}\label{sec:init-to-rad}

In this section, we provide an algorithm called \textsc{AlgoLSAtoRad} to transition from a constrained LSA w.r.t.~some site vector $\overline{s} \in \mathbb{R}^{d \cdot k}$ to a radial clustering w.r.t.~$\overline{s}$.
This algorithm can be seen as a pre-processing step for the linear transition along radial clusterings from an initial site vector $s$ to a target site vector~$t$, which is discussed in Section~\ref{sec:rad-to-rad}. Recall that we do not have an explicit algebraic representation of the bounded-shape partition polytope $\mathcal{P}^{\pm}$, which is why we perform the computations over the corresponding bounded-shape transportation polytope $\mathcal{T}^{\pm}$. We enter the algorithm with cluster size bounds as part of the input; see lines 2 and 3 of Algorithm~\ref{alg:overalltransition}. 

To ensure separability of all intermediate clusterings in the transition, we repeat the following steps: the initial clustering $C$ is a vertex of the bounded-shape transportation polytope $\mathcal{T}^{\pm}$. First, we walk along an improving edge in $\mathcal{T}^{\pm}$ to an adjacent vertex. This vertex corresponds to a better (w.r.t.~the objective vector $c(\overline{s})$) clustering $C'$ of different shape. (Note $C$ already was a constrained LSA, i.e., optimal over all clusterings of the same shape.) Then, we fix the lower and upper bounds on the cluster sizes to $\abs{C'}$ and transition in the corresponding single-shape transportation polytope $\mathcal{T}^=(\abs{C'})$ until we reach an optimal clustering with shape $\abs{C'}$ in direction of~$c(\overline{s})$, i.e., a constrained LSA w.r.t.~$\overline{s}$ for shape $\abs{C'}$. As we will show in Lemma~\ref{lem:transition:init-to-radial} and explain in the paragraph following it, there exists such an optimal clustering that differs from the initial clustering $C$ by a single sequential exchange, and it is easy to devise it from $C'$. Not only does this mean we take a step of the desired form, but it also makes it possible to compute a shared separating power diagram for two consecutive clusterings; we prove this as part of Theorem \ref{thm:init-to-rad}. This concludes the first iteration. 

\begin{figure}
\centering
\subcaptionbox{Initial LSA}{\includegraphics[width=0.3\textwidth]{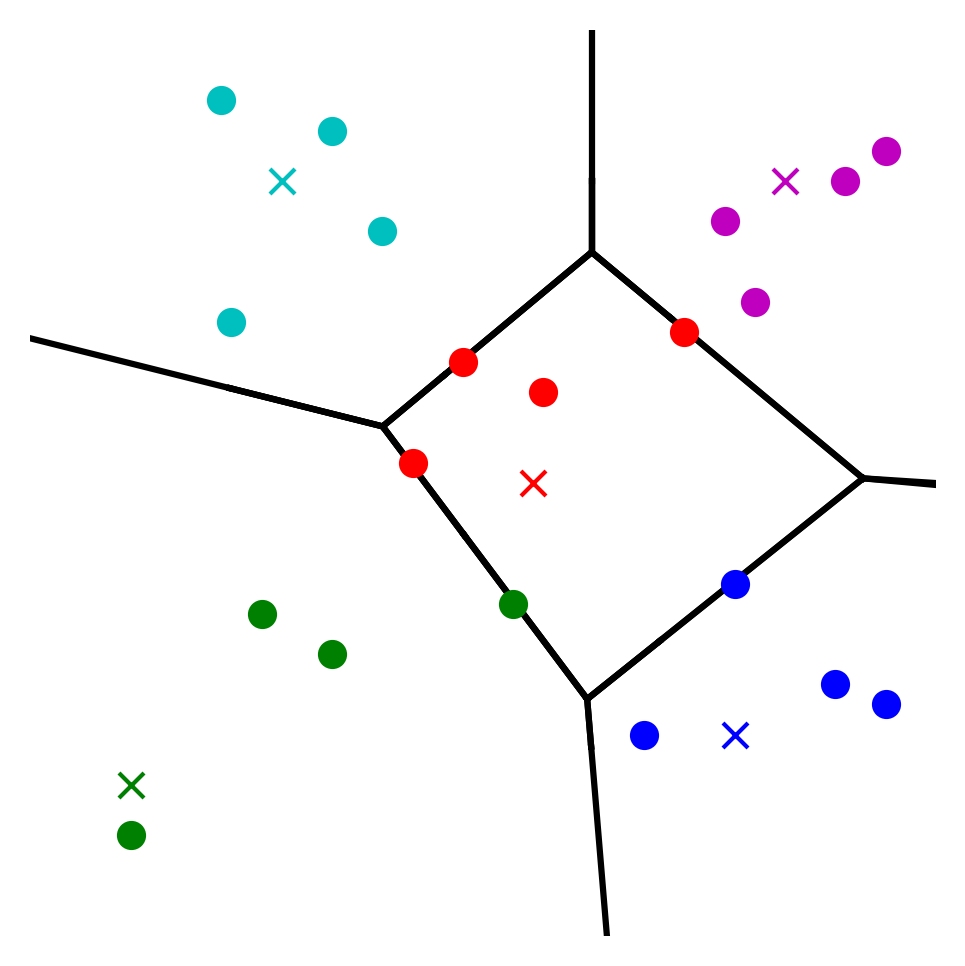} }%
\hfill
\subcaptionbox{Intermediate LSA }{\includegraphics[width=0.3\textwidth]{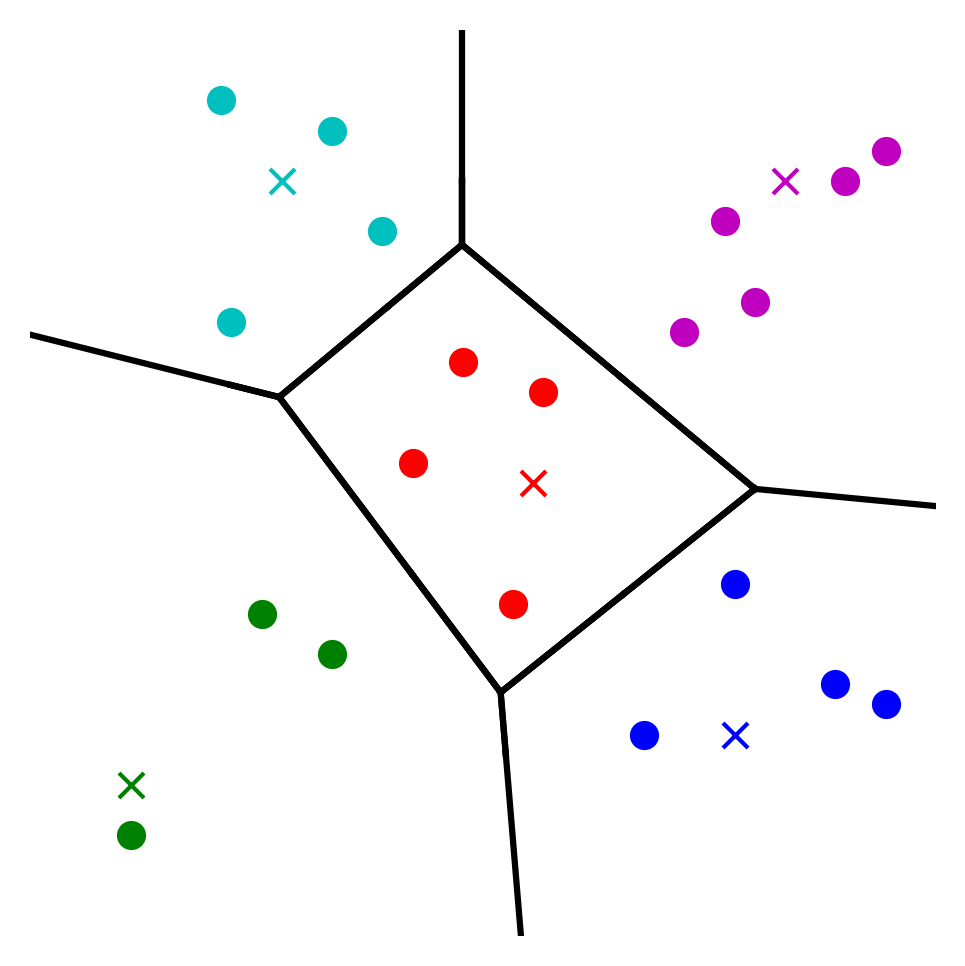}}%
\hfill
\subcaptionbox{Target LSA }{\includegraphics[width=0.3\textwidth]{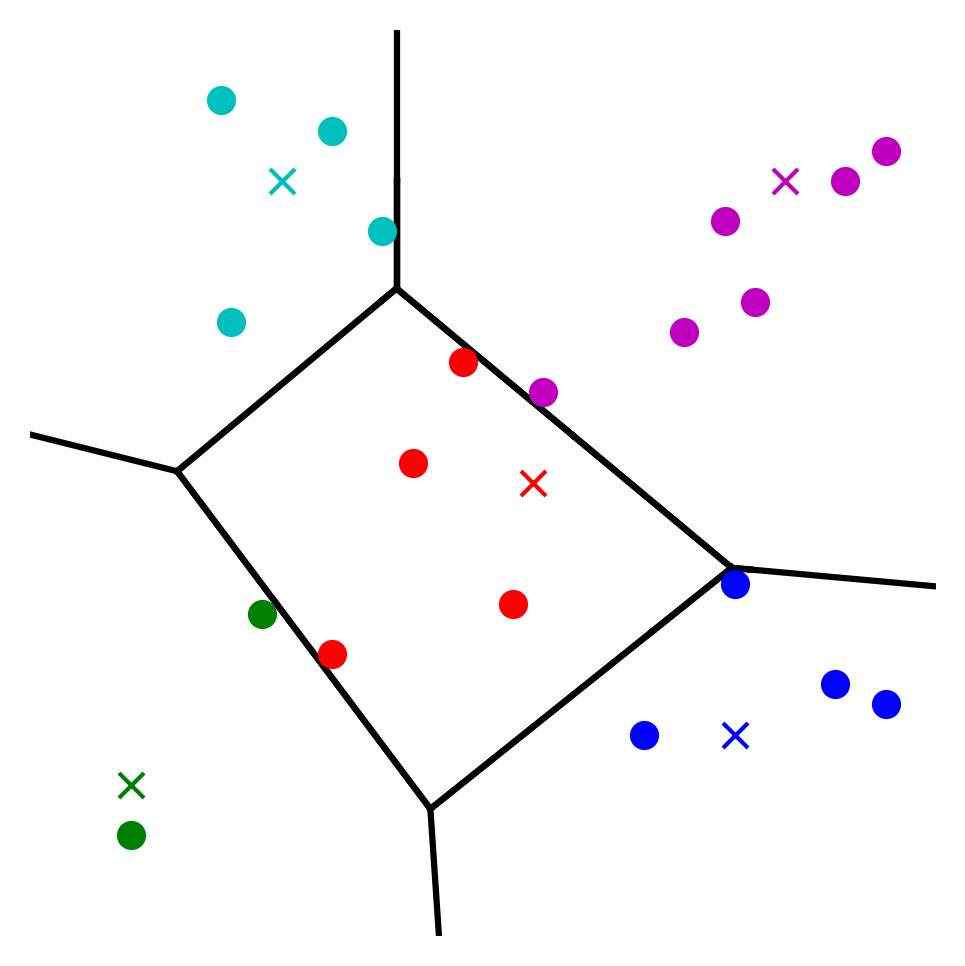}}%
\caption{Initial, target, and one intermediate LSA in a fixed-site transition as in \textsc{AlgoLSAtoRad}. The target LSA is a radial clustering for the (implicit) cluster size bounds. The three clusterings are constrained LSAs for the same sites, but different shapes.}
\label{fig:startInto2Rad}
\end{figure}

We update the initial clustering from $C$ to $C'$ and repeat this scheme until we reach a clustering that maximizes $c(\bar s)$ over $\mathcal{T}^{\pm}$, i.e., the desired radial clustering. See \textsc{AlgoLSAtoRad}, Algorithm~\ref{alg:init-to-rad}, for a description of this scheme in pseudocode. 
Figure~\ref{fig:startInto2Rad} depicts an initial and target clustering, as well as an intermediate clustering in the transition. Note that each clustering $C^r$ ($r \geq 1$) is a constrained LSA w.r.t.~$\overline{s}$ for its own shape~$\abs{C^r}$, as we optimize over the single-shape transportation polytope in line~\ref{alg:init-to-rad:single-shape-step} in direction of $c(\overline{s})$ in each iteration. The constructed sequence is the desired fixed-site transition. We sum up the favorable properties of the algorithm, and its output, in the following theorem. 

\IncMargin{1em}
\begin{algorithm}[t]

	\Indm
	\KwIn{Constrained LSA $C$ w.r.t.~a site vector $\overline{s} \in \mathbb{R}^{d \cdot k}$; cluster size bounds $\kappa^\pm$}
	\KwOut{Sequence of constrained LSAs and sequence of corresponding power diagrams that satisfy the properties of Theorem \ref{thm:init-to-rad}}
	\Indp
	
	\medskip
	
	$r \leftarrow 0$ and $C^0 \leftarrow C$\;
	\While{$C^r$ is not optimal in $\mathcal{T}^{\pm}(\kappa^-,\kappa^+)$ in direction of $c(\overline{s})$}{
	    Perform a simplex step from $C^r$ w.r.t.~objective $\max c(\overline{s})^Ty$ in $\mathcal{T}^{\pm}(\kappa^-,\kappa^+)$ to an adjacent vertex$\slash$clustering; update $C$ to the new clustering\;\label{alg:init-to-rad:bounded-shape-step}
	    Find an optimal clustering in $\mathcal{T}^=(\abs{C})$ that differs from $C^r$ by a single sequential exchange; update $C$ to the new clustering\;\label{alg:init-to-rad:single-shape-step}
	    Solve LP~(\ref{LP:computeweights}) with $\overline{s}$ and let $P^r$ be the corresponding power diagram\;
	    \If{$r \geq 1$}{
	        Compute a feasible solution for~(\ref{LP:computeweights:boundarypowerdiagram}) with $\overline{s}$ and $C^\nu=C$, $C^{\nu-1}=C^{r-1}$, and let $\overline{P}^r$ be the corresponding power diagram\;
	    }
	    $r \leftarrow r+1$ and $C^{r} \leftarrow C$\;
	}
	
    \Return sequence of clusterings\vspace*{-0.25cm} $$(C^{0},C^{1},\dots,C^{r})$$ \hspace*{1.18cm}and sequence of power diagrams\vspace*{-0.25cm} $$(P^0,\overline{P}^1,P^1,\dots,\overline{P}^r,P^r)\;$$

\caption{\textsc{AlgoLSAtoRad}($C,\overline{s},\kappa^-,\kappa^+$). Fixed-site transition from initial LSA w.r.t.~site vector~$\overline{s}$ to corresponding radial clustering.}
\label{alg:init-to-rad}
\end{algorithm}
\DecMargin{1em}

\begin{theorem}\label{thm:init-to-rad}
Let $C^{\overline{s}}$ be a constrained LSA with site vector $\overline{s}$, and let $\kappa^-,\kappa^+$ be cluster size bounds such that $\kappa^- \leq |C| \leq \kappa^+$. \textsc{AlgoLSAtoRad} returns a sequence of clusterings $$(C^{0},C^{1},\dots,C^{r})$$ and a sequence of power diagrams $$(P^0,\overline{P}^1,P^1,\dots,\overline{P}^r,P^r)$$ that satisfy the following properties:
\begin{enumerate}
    \item $C^{0}=C^{\overline{s}}$; $C^r$ is a radial clustering for sites $\overline{s}$
    \item all clusterings are constrained LSAs for sites $\overline{s}$
    \item all clusterings $C$ have cluster sizes $|C_i|$ satisfying $ \kappa_i^- \leq |C_i| \leq \kappa_i^+$ for all $i\leq k$
    \item consecutive clusterings differ by a single sequential exchange of items
    \item $P^{i}$ is a separating power diagram for sites $\overline{s}$ for $C^{i}$
    \item for each pair $C^{i-1}$, $C^{i}$ of consecutive clusterings, $\overline{P}^{i}$ is a shared power diagram for sites $\overline{s}$ 
    \item the shapes $|C^{i}|$ are all distinct; the number of clusterings in the sequence is bounded by the number of shapes.
\end{enumerate}
\end{theorem}

\begin{proof}{Proof.}
Recall the informal description of \textsc{AlgoLSAtoRad} above. Most of the claimed properties are direct consequences of the design of the algorithm:

The initial clustering $C^0=C$ is a constrained LSA w.r.t.~$\overline{s}$. This is equivalent to $C^0$ being optimal over $\mathcal{T}^=(\abs{C^0})$, the single-shape transportation polytope of all clusterings of the same shape. Each other clustering $C$ in the produced sequence is computed as an optimum w.r.t.~$c(\overline{s})$ over $\mathcal{T}^=(\abs{C})$ (line $4$), and thus also is a constrained LSA w.r.t.~$\bar s$.

The improving simplex step in line $3$ guarantees that each consecutive clustering $C$ is {\em strictly} better w.r.t.~$c(\overline{s})$. Each clustering is an optimum w.r.t.~$c(\bar s)$ over $\mathcal{T}^=(\abs{C})$, so this strict improvement can only come from finding a clustering of a new shape that is distinct from the shapes of all previous clusterings in the sequence. Thus, the sequence of clusterings is finite; in fact, it is bounded by the number of possible shapes. The algorithm terminates with a radial clustering $C^r$ after $r$ runs of the while-loop. This proves properties $1$, $2$, and $7$. 

Further, line $3$ is the only step that can change the clustering shape. As the simplex step is taken over $\mathcal{T}^{\pm}(\kappa^-,\kappa^+)$, cluster sizes $|C_i|$ are bounded throughout by $ \kappa_i^- \leq |C_i| \leq \kappa_i^+$ for all $i\leq k$. This gives property $3$.

We will show property 4, i.e., the fact that two consecutive clusterings only differ by a single sequential exchange of items, in Lemma \ref{lem:transition:init-to-radial} below. 

Property $5$ follows immediately from each clustering $C^i$ in the sequence being optimal for sites $\overline{s}$ over $\mathcal{T}^=(\abs{C^i})$. LP~(\ref{LP:computeweights}) has a feasible solution and returns the desired $P^i$.

It remains to prove property $6$. Let $C^{i-1}$ and $C^i$ denote two consecutive clusterings. Recall that a clustering $C$ allows a separating power diagram for sites $\overline{s}$ if the vector $\overline{s}$ lies in the normal cone of the clustering vector $w(C)$ in any partition polytope \citep{Borgwardt2010}. In fact, this property also holds for {\em fractional} clusterings, where items can be assigned partially to multiple clusters (adding up to $1$) and the corresponding partition polytopes \citep{BriedenGritzmann2012}.
 
 We have to show that there exists a {\em shared} power diagram that induces both $C^{i-1}$  and $C^i$ for sites $\overline{s}$. Let $y^{i-1}$ and $y^i$ be the 0/1-vectors corresponding to $C^{i-1}$ and $C^i$. Recall that $C^{i-1}$ and $C^i$ only differ by a single sequential exchange of items (property 4). Consider the fractional clustering $C^\text{frac}$ corresponding to $y^\text{frac}=\frac{1}{2} (y^{i-1} + y^i)$. The vector $y^\text{frac}$ has components $0$, $0.5$, or $1$, and is informally obtained from applying `half' of the sequential exchange applied to $C^{i-1}$: each item moved now belongs $50\%$ to its original cluster in $C^{i-1}$ and $50\%$ to its new cluster in $C^i$ -- the corresponding variables in $y^\text{frac}$ are $0.5$.
 
 Let $\theta^{\pm}$ and $\eta^+$ be cluster size bounds derived as $\theta_l^-= \min\set{\abs{C^{i-1}_l},\abs{C^i_l}}$, $\theta_l^+= \max\set{\abs{C^{i-1}_l},\abs{C^i_l}}$, and $\eta_l^+= \max\set{\abs{C^{i-1}_l},\abs{C^\text{frac}_l}}$ for all $l \in [k]$. Recall $C^{i-1}$ is optimal w.r.t~$\overline{s}$ over $\mathcal{P}(|C^{i-1}|)$,
 As $C^i$ is optimal w.r.t.~$\overline{s}$ over $\mathcal{P}^{\pm}(\theta^-,\theta^+)$, $C^\text{frac}$ is optimal w.r.t~$\overline{s}$ over $\mathcal{P}^{\pm}(\theta^-,\eta^+)$. In fact, the same sequential exchange used to obtain $C^i$ from $C^{i-1}$ (property $4$) leads to the new optimal fractional clustering $C^\text{frac}$ when moving only $50\%$ of each item involved. This implies the existence of a separating power diagram $\overline{P}^i$ for sites $\overline{s}$ for the fractional clustering $C^\text{frac}$. This power diagram $\overline{P}^i$ induces both $C^{i-1}$ and $C^i$, so it is the desired shared power diagram. For consecutive clusterings $C^{i-1}$ and $C^{i}$, such a power diagram $\overline{P^i}$ is found through a solution of LP~(\ref{LP:computeweights:boundarypowerdiagram}) with sites $\overline{s}$ and $\nu = i$, and denoted in lines $6$ to $8$ of \textsc{AlgoLSAtoRad}. This proves the claim.
\hfill \qed
\end{proof}

To complete the proof of Theorem \ref{thm:init-to-rad}, it remains to show correctness of property $4$. We do so in the following lemma. In particular, it shows that the choice of $C$ in line~\ref{alg:init-to-rad:single-shape-step} is well-defined: there always exists an optimal constrained LSA for the new clustering shapes identified in line $3$ such that the difference to the previous clustering is only a single sequential exchange. 

\begin{lemma}\label{lem:transition:init-to-radial}
Let $\overline{s} \in \mathbb{R}^{d \cdot k}$ be a site vector and consider the $i$-th iteration of \textsc{AlgoLSAtoRad}. Let $C^{i-1}$ refer to the constrained LSA from the previous iteration, and let $C$ correspond to the clustering of updated shape (line $3$ of \textsc{AlgoLSAtoRad}). There exists a constrained LSA $C^i$ w.r.t.~$\overline{s}$ in $\mathcal{T}^=(\abs{C})$ such that $C^{i-1}$ and $C^i$ differ by a single sequential exchange of items.
\end{lemma}
\begin{proof}{Proof.}
For a simple wording, we indicate that a clustering is optimal over some transportation polytope w.r.t~$c(\overline{s})$ by saying it is optimal over the polytope.

We begin the $i$-th iteration of \textsc{AlgoLSAtoRad} with $C^{i-1}$, which is optimal over $\mathcal{T}^=(\abs{C^{i-1}})$. Clustering $C$, as devised in line $3$ of the algorithm, is of different shape and of strictly better objective function value than $C^{i-1}$. Let $y^{i-1}$ and $y$ be the 0/1 vectors in $\mathcal{T}^{\pm}$ corresponding to $C^{i-1}$ and $C$, respectively. By Proposition~\ref{prop:adjacentvertices}, we get that $CDG(C^{i-1},C)$ contains a single path $P$, which, without loss of generality, starts at cluster $1$ and ends at cluster $k$. 
If $C$ is optimal over $\mathcal{T}^=(\abs{C})$, we set $C^i=C$ and are done. 

So suppose $C$ is not optimal over $\mathcal{T}^=(\abs{C})$. This means there exists a different optimal clustering $C^\text{opt}$ over $\mathcal{T}^=(\abs{C})$ with $\abs{C^\text{opt}}=\abs{C}$ and corresponding $y^\text{opt}$ satisfying $c(\overline{s})^T(y^\text{opt} - y) > 0$.
Consider $CDG(C^{i-1},C^\text{opt})$. Note all nodes in $\set{2,\dots,k-1}$ have even degree and the nodes $1$ and $k$ have odd degree. Hence, $CDG(C^{i-1},C^\text{opt})$ greedily decomposes into a path $\overline{P}$ from node $1$ to $k$ and a set of arc-disjoint cycles $\mathcal{CY}$ \citep{Borgwardt2013,BorgwardtViss2020}.

Let $y(\mathcal{CY}):=y^\text{opt} - y^i \in \{-1,0,1\}^{n \cdot k}$; it has components $y_{lj} = -1$ and $y_{\ell j} = 1$ if and only if there is a cycle in the set $\mathcal{CY}$ that contains an arc $(l,\ell)$ with label $x_j$. 
We use the same notation $y(\mathcal{CY}_v)$ for single cycles $\mathcal{CY}_v$ in $\mathcal{CY}$ and the path $\overline{P}$, too.  The entries of $y(\mathcal{CY})$, $y(\mathcal{CY}_v)$ or $y(\overline{P})$, represent the removal of item $x_j$ from cluster $C^{i-1}_l$ and addition of it to cluster $C_\ell^\text{opt}$.

Note that $c(\overline{s})^T y(\mathcal{CY}_t)\leq 0$ for any cycle $\mathcal{CY}_t$ from $\mathcal{CY}$; otherwise $C^{i-1}$ would not be optimal over $\mathcal{T}^=(\abs{C^{i-1}})$. Thus, we have $c(\overline{s})^T y(\mathcal{CY})\leq 0$. By applying the sequential exchange corresponding to $\overline{P}$ to $C^{i-1}$, we obtain a clustering $C^i$ with $\abs{C^i} = \abs{C^\text{opt}}$ and corresponding $y^i$ satisfying
\begin{equation*}
c(\overline{s})^T (y^\text{opt} - y^{i-1}) = c(\overline{s})^T (y(\overline{P}) + y(\mathcal{CY})) \leq c(\overline{s})^T y(\overline{P}) = c(\overline{s})^T (y^i - y^{i-1})
\end{equation*}
and thus
\begin{equation*}
c(\overline{s})^T (y^i - y^\text{opt}) \geq 0
\end{equation*}
As $C^\text{opt}$ was optimal over $\mathcal{T}^=(\abs{C})$, $C^i$ is optimal, too; the inequality is satisfied with equality. Clustering $C^i$ differs from $C^{i-1}$ by a single sequential exchange. This proves the claim.\hfill \qed
\end{proof}

The proof shows that finding the next clustering $C^i$ in the sequence (i.e., line~\ref{alg:init-to-rad:single-shape-step} of \textsc{AlgoLSAtoRad}) can be done by computing an optimal clustering $C^\text{opt}$ using the simplex method over $\mathcal{T}^=(\abs{C})$ (starting from $C$ from line \ref{alg:init-to-rad:bounded-shape-step}), then setting up $CDG(C^{i-1},C^\text{opt})$ and greedily deleting all cycles from it, and finally applying the single sequential exchange corresponding to the only remaining path $\overline{P}$ in the CDG to clustering $C^{i-1}$. 


Summing up, \textsc{AlgoLSAtoRad} finds the desired fixed-site transition from a constrained LSA to a radial clustering, which is used as the first and final part of the overall transition in Theorem \ref{alg:overalltransition} and Algorithm~\ref{alg:overalltransition}.

\section{A Linear Transition between Radial Clusterings}\label{sec:rad-to-rad}

In this section, we devise an algorithm called \textsc{AlgoRadToRad} to compute a transition between two radial clusterings $C^s_\text{rad}$ and $C^t_\text{rad}$. Each step of the transition corresponds to a single exchange of items, and each intermediate clustering is a radial clustering itself. 

Geometrically, both $C^s_\text{rad}$ and $C^t_\text{rad}$ lie on the boundary of $\mathcal{P}^\pm$ (for cluster size bounds $\kappa^{\pm}$ induced by $C^s_\text{rad}$ and $C^t_\text{rad}$) and the transition forms a walk along the boundary of this polytope. The sequence of clusterings follows the linear transition $(1-\lambda) s +  \lambda t$ of site vectors for increasing $\lambda$. In the corresponding $\mathcal{T}^\pm$, $C^s_\text{rad}$ and $C^t_\text{rad}$ are vertices and the transition takes the form of an edge walk, following a linear transition of objective functions $(1-\lambda)c(s) + \lambda c(t)$ for increasing $\lambda$. For this transition, we need to identify when and how clusterings change. 

By performing our computations over $\mathcal{T}^\pm$ (for which there exists an explicit representation), this can be done by adapting classical tools of sensitivity analysis and ranging. Generally, vertices of $\mathcal{T}^{\pm}$ can be highly degenerate. As a service to the reader, we recall how ranging works for degenerate vertices in Appendix \ref{app:ranging}. In particular, it is possible to provide an explicit range for $\lambda$ for which a current optimal vertex, not only basis, remains optimal.

Let us adapt this to an iterative scheme over $\mathcal{T}^\pm$. We represent the feasible set~(\ref{transport}) in standard form through the addition of some slack variables for constraints (\ref{transport:lowerbound}) and (\ref{transport:upperbound}). All feasible clusterings correspond to vertices of $\mathcal{T}^{\pm}$, and $C^s_\text{rad}$ is optimal w.r.t.~$c(s)$. The desired transition $(1-\lambda)c(s) + \lambda c(t)$ for increasing $\lambda$ can equivalently be written in the form $c(s) + \lambda \Delta c$ for $\Delta c = c(t) - c(s)$. Then LP (\ref{dualradialLP}) in Appendix \ref{app:ranging} can be applied to compute the {\em breakpoint} $\lambda > 0$ beyond which the current vertex is not optimal anymore. An update of the underlying clustering happens and we take a step along an edge of $\mathcal{T}^{\pm}$ to a new, adjacent vertex.


We repeat this analysis iteratively, updating to the new vertex and an associated basis, to identify a sequence of breakpoints in the transition from $c(s)$ to $c(t)$. The result is a variant of a parametric linear programming algorithm over $\mathcal{T}^{\pm}$. See \textsc{AlgoRadToRad}, Algorithm~\ref{alg:rad-to-rad}, for a description of the scheme in pseudocode. The algorithm also performs the computation of a sequence of power diagrams associated to the clusterings. Figure~\ref{fig:radialTransitions} depicts an initial and target clustering, as well as an intermediate clustering in the transition. 

 \textsc{AlgoRadToRad} is phrased in reference to the normal cones of consecutive vertices of $\mathcal{T}^{\pm}$. Note that $\lambda_r$ ($1 \leq r \leq m$) indicates the breakpoint where the objective vector $c((1-\lambda_r)s + \lambda_r t)$ lies in the intersection of the normal cones of $C^{\lambda_{r-1}}$ and $C^{\lambda_r}$ w.r.t.~$\mathcal{T}^\pm$. Each clustering in the sequence is distinct from its predecessor, and we perform an edge walk over $\mathcal{T}^\pm$. At the end of this section, we show that the corresponding clustering vectors are distinct, too, so that we obtain a proper walk along boundary points over $\mathcal{P}^\pm$, as well.  

\begin{figure}
\centering
\subcaptionbox{Initial radial clustering}{\includegraphics[width=0.3\textwidth]{5ptSimpleStartLSA.png}}%
\hfill
\subcaptionbox{\small Intermediate radial clustering}{\includegraphics[width=0.3\textwidth]{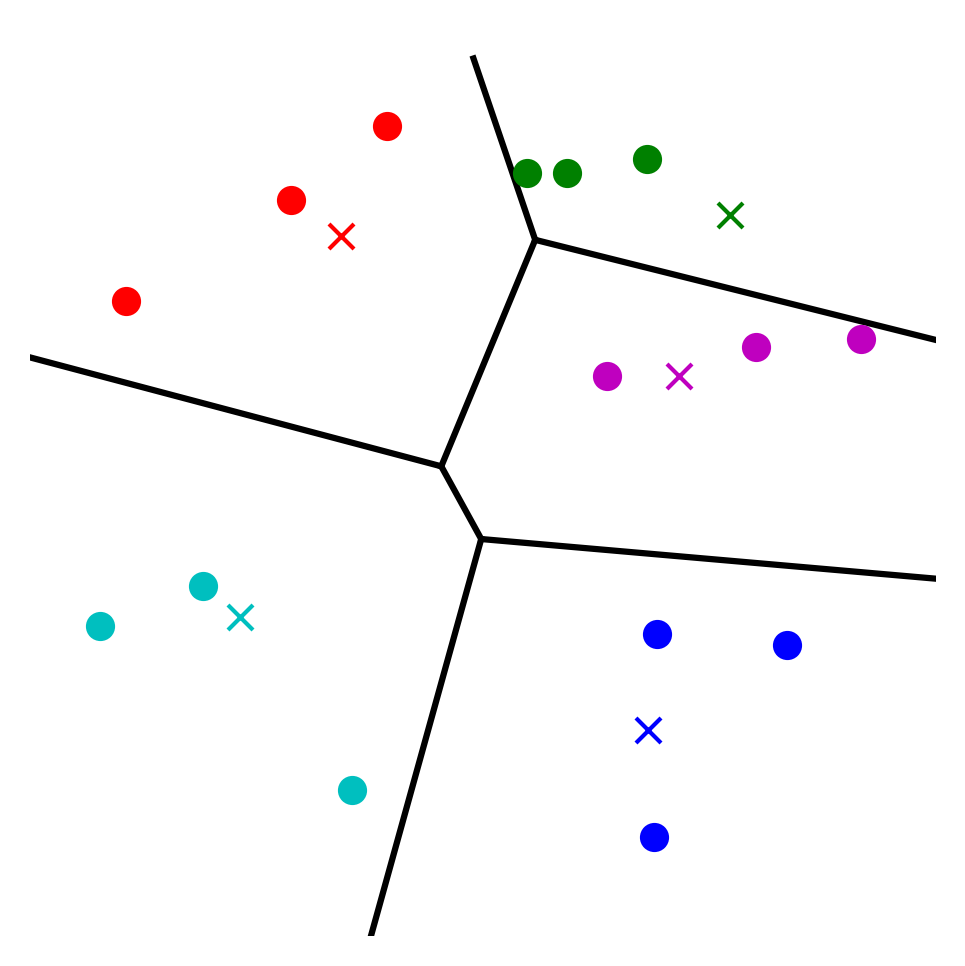}}%
\hfill
\subcaptionbox{Target radial clustering}{\includegraphics[width=0.3\textwidth]{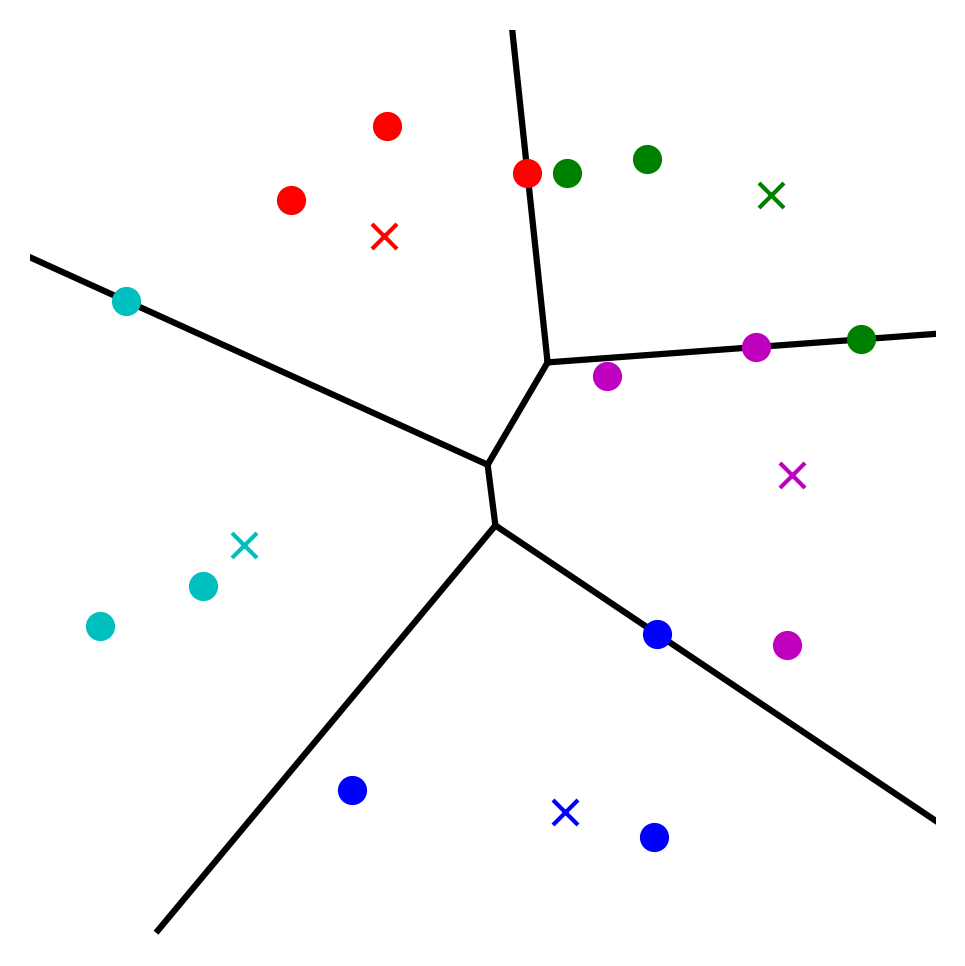}}%
\caption{Initial, target, and one intermediate radial clustering in a linear transition as in \textsc{AlgoRadToRad}. The sites follow a linear transition from $s$, for the initial clustering, to $t$, for the target clustering. The sites for the intermediate clustering lie on the line segment between their initial and target positions.}
\label{fig:radialTransitions}
\end{figure}

We prove correctness of \textsc{AlgoRadToRad}, along with a number of favorable properties of the transition and output, in the following theorem. 

\IncMargin{1em}
\begin{algorithm}[t]

	\Indm
	\KwIn{Radial clusterings $C^s_\text{rad}$ and $C^t_\text{rad}$ w.r.t.~site vectors $s \in \mathbb{R}^{d \cdot k}$ and $t \in \mathbb{R}^{d \cdot k}$;\\ \hspace*{1.26cm} cluster size bounds $\kappa^\pm$}
	\KwOut{Sequence of radial clusterings and sequence of corresponding power diagrams that satisfy the properties of Theorem \ref{thm:rad-to-rad}}
	\Indp
	
	\medskip
	
	Set $\lambda \leftarrow 0, \lambda_0 \leftarrow 0, r \leftarrow 0$ and $C^{\lambda_0} = C^s_\text{rad}$\;
	\While{$C^{\lambda_r} \not= C^t_\text{rad}$}{
	    Increase $\lambda$ to the next ranging breakpoint such that $c(s^\lambda)=\lambda c(t) + (1- \lambda) c(s)$ satisfies $c(s^\lambda) \in N_{\mathcal{T}^{\pm}(\kappa^-,\kappa^+)}(C^{\lambda_r}) \cap N_{\mathcal{T}^{\pm}(\kappa^-,\kappa^+)}(C^\lambda)$ for some adjacent vertex$\slash$clustering $C^\lambda \neq C^{\lambda_r}$\;\label{alg:rad-to-rad:update} 
	    Set $r \leftarrow r + 1$, $\lambda_r \leftarrow \lambda$ and $C^{\lambda_r} \leftarrow C^\lambda$\;
	    
	    Compute a feasible solution for LP~(\ref{LP:computeweights:boundarypowerdiagram}) with $\overline{s}=s^{\lambda_r}$ and $C^{\nu-1}=C^{\lambda_{r-1}}$, $C^\nu=C^{\lambda_r}$, and let $\overline{P}^r$ be the corresponding power diagram\;
	    
	    \If{$r \geq 2$}{
	        Solve LP~(\ref{LP:computeweights}) with $\overline{s} = \frac{1}{2}(s^{\lambda_{r-1}} + s^{\lambda_r})$
	        and $C^\nu = C^{\lambda_{r-1}}$ and let $P^{\lambda_{r-1}}$ be the corresponding power diagram\;
	    }
	    
	}
	
    \Return sequence of clusterings\vspace*{-0.25cm} $$(C^{\lambda_0},C^{\lambda_1},\dots,C^{\lambda_r})$$ \hspace*{1.18cm}and sequence of power diagrams\vspace*{-0.25cm} $$(\overline{P}^{\lambda_1},P^{\lambda_1},\overline{P}^{\lambda_2},\dots,P^{\lambda_{r-1}},\overline{P}^{\lambda_r})\;$$
    
\caption{\textsc{AlgoRadToRad}($C^s_\text{rad},C^t_\text{rad},s,t,\kappa^-,\kappa^+$). Linear transition from radial clustering $C^s_\text{rad}$ w.r.t.~$s$ to radial clustering $C^t_\text{rad}$ w.r.t.~$t$.}
\label{alg:rad-to-rad}
\end{algorithm}
\DecMargin{1em}

\begin{theorem}\label{thm:rad-to-rad}
Let $C^s_\text{rad}$ and $C^t_\text{rad}$ be radial clusterings w.r.t.~site vectors $s$ and $t$. Further, let $\kappa^-,\kappa^+$ be cluster size bounds such that $\kappa^- \leq |C^s_\text{rad}|, |C^t_\text{rad}| \leq \kappa^+$, and let $s^\lambda = (1-\lambda)s + \lambda t$. \textsc{AlgoRadToRad} returns a sequence of clusterings $$(C^{\lambda_0},C^{\lambda_1},\dots,C^{\lambda_r})$$
and a sequence of power diagrams
$$(\overline{P}^{\lambda_1},P^{\lambda_1},\overline{P}^{\lambda_2},\dots,P^{\lambda_{r-1}},\overline{P}^{\lambda_r})$$ that satisfy the following properties:
\begin{enumerate}
    \item $C^{\lambda_0}=C^s_\text{rad}$, $C^{\lambda_r}=C^t_\text{rad}$
    \item the $C^{\lambda_i}$ are radial clusterings for sites $s^{\lambda_i}=(1-\lambda_i) s + \lambda_i t$ for all $i\geq 0$
    \item all clusterings $C$ have cluster sizes $|C_i|$ satisfying $ \kappa_i^- \leq |C_i| \leq \kappa_i^+$ for all $i\leq k$
    \item consecutive clusterings differ by a single cyclical or sequential exchange of items
    \item $P^{\lambda_i}$ is a separating power diagram for $C^{\lambda_i}$ for sites $\frac{1}{2}(s^{\lambda_{i-1}} + s^{\lambda_i})$
    \item for consecutive clusterings $C^{\lambda_{i-1}}$, $C^{\lambda_i}$, $\overline{P}^{\lambda_i}$ is a shared power diagram for sites $s^{\lambda_i}$.
    \end{enumerate}
\end{theorem}
\begin{proof}{Proof.}
Recall the informal description of \textsc{AlgoRadToRad} above. As before, most of the claimed properties are direct consequences of the design of the algorithm. 

Line $1$ is the initialization of the algorithm with $C^{\lambda_0} = C^s_\text{rad}$. Clustering $C^{\lambda_r}$ for increasing $r$ is the one being worked on. Lines $3-7$ are the steps in the main loop of the algorithm. They describe a step from the current clustering to a next one (lines 3 and 4), along with the computation of associated power diagrams (lines 5 to 8), until $C^{\lambda_r}$ becomes equivalent to $C^t_\text{rad}$ (line 2). The main idea is the creation of a sequence of clusterings $(C^{\lambda_0}=C^s_\text{rad},C^{\lambda_1},\dots,C^{\lambda_r}=C^t_\text{rad})$ for increasing $\lambda_i$, where each clustering $C^{\lambda_i}$ is a radial clustering for sites $s^{\lambda_i}=(1-\lambda_i) s + \lambda_i t$. 
Let us take a closer look at why such a sequence is created:

The algorithm keeps running as long as the current clustering is not $C^t_\text{rad}$ (line 2). The number $\lambda$ satisfies $\lambda=\lambda_{i-1}$ at the beginning of the $i$-th run of the while loop. In line $3$, we increase $\lambda$ until a ranging breakpoint is hit. Recall that a breakpoint is a value of $\lambda$ for which the current vertex and a new vertex are equally good with respect to the objective function $c(s^\lambda)$. As we perform the computations over $\mathcal{T}^{\pm}$, whose vertices can be highly degenerate, we solve LP (\ref{dualradialLP}).

A strict increase in $\lambda$ happens, which must lead to a new vertex$\slash$clustering: $c(s^\lambda)=\lambda c(t) + (1- \lambda) c(s)$ now is optimal over $\mathcal{T}^{\pm}(\kappa^-,\kappa^+)$ for both the current clustering $C^{\lambda_{i-1}}$ and a new, next clustering $C^\lambda$. Both clusterings are radial for the sites $s^\lambda$ over $\mathcal{P}^{\pm}(\kappa^-,\kappa^+)$. (In Lemma \ref{lem:transition:rad-to-rad} below, we prove that the two clusterings have distinct clustering vectors, too -- but this is not required for our arguments here.) The vectors $c(s^\lambda)$ and $s^\lambda$ are on the shared boundary of the normal cones for the clusterings in $\mathcal{T}^{\pm}(\kappa^-,\kappa^+)$ and $\mathcal{P}^{\pm}(\kappa^-,\kappa^+)$, respectively. The actual algebra to find a breakpoint is the solution of LP~(\ref{dualradialLP}) for the current clustering and $c=c(s^{\lambda_{i-1}})$; the result gives $\lambda-\lambda_{i-1}$. 

 In the $i$-th iteration, line 4 is an increase of the iteration counter, the labeling of the breakpoint $\lambda$ as $\lambda_i$, and the labeling of the new clustering $C^\lambda$ as $C^{\lambda_i}$. Lines 5 to 8 describe the computation of the associated power diagrams, which we turn to below. Then the next iteration of the loop is started.
The parameter $\lambda$ increases strictly throughout the algorithm and, at the end of the algorithm, it approaches $1$. There are only finitely many clusterings (or vertices of $\mathcal{T}^{\pm}$), so there exists a $\lambda_r \leq 1$ for which $C^{\lambda_r} = C^t_\text{rad}$, as $c(t) \in N_{\mathcal{T}^{\pm}}(C^t_\text{rad})$ by construction. The algorithm then terminates due to a successful check of $C^{\lambda_r} = C^t_\text{rad}$ (line 2). All computations are done over $\mathcal{T}^{\pm}=\mathcal{T}^{\pm}(\kappa^-,\kappa^+)$, so all clusterings in the returned sequence have cluster sizes satisfying $ \kappa_i^- \leq |C_i| \leq \kappa_i^+$ for all $i\leq k$. This proves properties $1$, $2$, and $3$. 

The ranging-based computation of a breakpoint and adjacent {\em vertex} (not only basis) of $\mathcal{T}^{\pm}$ means that two consecutive clusterings are connected by a shared edge. They differ by a single cyclical or sequential exchange of items; see Proposition \ref{prop:adjacentvertices}. This gives property 4. Overall, this proves that the algorithm terminates with a sequence of clusterings $(C^{\lambda_0},C^{\lambda_1},\dots,C^{\lambda_r})$ of the claimed properties.

It remains to prove properties 5 and 6, i.e., to prove that the  constructed sequence of power diagrams $(\overline{P}^{\lambda_1},P^{\lambda_1},\overline{P}^{\lambda_2},\dots,P^{\lambda_{r-1}},\overline{P}^{\lambda_r})$
satisfies the claimed properties. In Section \ref{sec:powerdiagram}, we described the LPs to find a separating power diagram for given sites and a given clustering and to find a shared power diagram for given sites and a pair of clusterings. We have to show that our input allows for feasible solutions.

In the $i$-th iteration, the computation of a feasible solution for LP~(\ref{LP:computeweights:boundarypowerdiagram}) in line 5 refers to the computation of a shared power diagram $\overline{P}^{\lambda_i}$ for sites $\overline{s}=s^{\lambda_i}$ and the consecutive clusterings $C^\nu=C^{\lambda_i}$ and $C^{\nu-1}=C^{\lambda_{i-1}}$. The existence of a feasible solution, i.e., of a shared power diagram, follows from $c(s^{\lambda_i})$ being in the intersection of the normal cones $c(s^{\lambda_i}) \in N_{\mathcal{T}^{\pm}}(C^{\lambda_{i-1}}) \cap N_{\mathcal{T}^{\pm}}(C^{\lambda_i})$ (line 3). Recall that $C^{\lambda_{i-1}}$ and $C^{\lambda_i}$ correspond to adjacent vertices in $\mathcal{T}^{\pm}$. Geometrically, any point on the edge between the vertices (and possibly a face of higher dimension) of $\mathcal{T}^{\pm}$ contains $c(s^{\lambda_i})$ in its normal cone. Equivalently, there is a face of $\mathcal{P}^{\pm}$ of at least dimension $1$ with $s^{\lambda_i}$ in its normal cone. Similar to the proof of Theorem \ref{thm:init-to-rad}, the fractional clusterings corresponding to points in this face allow a separating power diagram that induces both $C^{\lambda_{i-1}}$ and $C^{\lambda_{i}}$ -- the desired shared power diagram $\overline{P}^{\lambda_i}$. This shows property 6.

Lines 6 to 8 describe the computation of a `good' separating power diagram, maximizing the margin, for all intermediate clusterings. The power diagram $P^{\lambda_{i-1}}$ is computed for $C^{\lambda_{i-1}}$ in the $i$-th iteration (for $i\geq 2$). It is performed through a solution of LP~(\ref{LP:computeweights}) for sites $\overline{s}= \frac{1}{2}(s^{\lambda_{i-1}} + s^{\lambda_i})$ and $C^\nu = C^{\lambda_{i-1}}$. Note that information on $\lambda_i$ is required for the definition of $\overline s$, which is why the computation of $P^{\lambda_{i-1}}$ for $C^{\lambda_{i-1}}$ appears at the end of the iteration that found the next $C^{\lambda_{i}}$. 

The site vector $\overline{s}$ lies on the line segment between the sites $s^{\lambda_{i-1}}$ at the breakpoints when transitioning to $C^{\lambda_{i-1}}$ and the sites $s^{\lambda_i}$ when transitioning away from it, to $C^{\lambda_{i}}$. Thus, the vector $\overline{s}$ lies in the normal cone of $w(C^{\lambda_{i-1}})$, and $c(\overline{s})\in N_{\mathcal{T}^{\pm}(\kappa^-,\kappa^+)}(C^{\lambda_{i-1}})$. This implies that $C^{\lambda_{i-1}}$ allows a separating power diagram for sites $\overline{s}$, and such a power diagram is found as a feasible solution to LP~(\ref{LP:computeweights}). This shows property 5.

Summing up, the returned sequence $(\overline{P}^{\lambda_1},P^{\lambda_1},\overline{P}^{\lambda_2},\dots,P^{\lambda_{r-1}},\overline{P}^{\lambda_r})$ is an alternating sequence of shared power diagrams and good separating power diagrams following the sequence of clusterings and satisfying properties 5 and 6. This completes the proof.
\hfill \qed
\end{proof}


We conclude this section by showing that \textsc{AlgoRadToRad}, which is designed to construct a sequence of clusterings through an edge walk over a bounded-shape transportation polytope, also computes a (well-defined) walk along the boundary of the corresponding partition polytope. More specifically, we prove that consecutive clusterings in the returned sequence have distinct clustering vectors. Thus, each step along an edge in the transportation polytope corresponds to a proper step along the boundary of the partition polytope. 

\begin{lemma}\label{lem:transition:rad-to-rad}
Let $(C^{\lambda_0},C^{\lambda_1},\dots,C^{\lambda_r})$ be a sequence of clusterings returned by \textsc{AlgoRadToRad}. 
For every $1 \leq i \leq r$, $w(C^{\lambda_i}) \neq w(C^{\lambda_{i-1}})$.
\end{lemma}
\begin{proof}{Proof.}
Line~\ref{alg:rad-to-rad:update} of \textsc{AlgoRadToRad} ensures that $C^{\lambda_{i-1}} \not= C^{\lambda_i}$ for all $i \geq 1$. Further, the corresponding vertices $y^{\lambda_{i-1}}$ and $y^{\lambda_i}$ of $\mathcal{T}^{\pm}$ share an edge. By Proposition~\ref{prop:adjacentvertices}, the clustering difference graph $CDG(C^{\lambda_{i-1}},C^{\lambda_i})$ contains a single cycle or path. For any fixed node $l \in [k]$ on this cycle or path, let $x^+_l$ and $x^-_l$ be the label of the arc that enters and leaves node $l$ in the CDG, respectively.

Since the items in data set $X\subset \mathbb{R}^d$ are distinct, i.e., we have $x^+_l \not= x^-_l$. The component of the vector $w(C^{\lambda_i}) - w(C^{\lambda_{i-1}})$ corresponding to cluster $l$ equals $\sum_{x \in C^{\lambda_i}_l } x - \sum_{x \in C^{\lambda_{i-1}}_l} x  = x^+_l - x^-_l \neq 0$, i.e., $w(C^{\lambda_i}) - w(C^{\lambda_{i-1}}) \neq 0$. This implies $w(C^{\lambda_i}) \neq w(C^{\lambda_{i-1}})$.
\hfill \qed
\end{proof}

\section{Conclusion}\label{sec:conclusion}

In this paper, we designed an algorithm to transition between two given constrained LSAs in the form of a sequence of clusterings that satisfies the many favorable properties listed in Theorem \ref{thm:overalltransition}. We would like to highlight a few natural questions to study for further improvements.

First, a study and optimization of the efficiency of our methods would be of interest. While most of the applications that we have encountered were not sensitive to computation times for the transition, of course the scalability of our methods is of interest. In Appendix \ref{app:computations}, we show some running times for our proof-of-concept implementation. As we use steps of the simplex method, and information from an optimal simplex tableau, its bottleneck lies in the setup and solution of the underlying LPs. These LPs are simple generalizations of classical transportation problems, and as such are well-understood. Essentially, we can scale to problem sizes for which transportation problems are still solvable. In addition to a bound based on LP theory, it would be promising to study whether some of the LP-based steps of our algorithm (such as lines 3 and 4 in \textsc{AlgoLSAtoRad} and line 3 in \textsc{AlgoRadToRad}) can be replaced by more efficient combinatorial algorithms. 

Second, related to this is a desire for a bound on the number of steps of the transition. In particular, we are interested in the number of steps for \textsc{AlgoRadToRad}, where radial clusterings are traversed. This part of the transition corresponds to an edge walk over $\mathcal{T}^\pm$, so that some first bounds come from the combinatorial diameter or so-called circuit diameter of these polytopes; see \cite{Borgwardt2013,bv-19a}. However, previous literature only takes into account the assignment of items to clusters, and not the locations of items or sites in the underlying space. For the fixed-site transition from a general LSA to its radial counterpart, we only have a bound in the form of the number of feasible clustering shapes; see Theorem \ref{thm:init-to-rad}, property $7$. Again, this bound does not take into account any geometric information. In the computations for Appendix \ref{app:computations}, we have seen that, in practice, the number of steps required is much lower than the number of shapes for the fixed-site transitions. There is plenty of room for improved theoretical bounds in both cases through a use of the locations of items and sites.

Third, there are several ways to try to improve on the design of the current approach. We broke up the walk into three parts. Essentially, the walk from an initial LSA to its corresponding radial clustering (and from the target radial clustering to target LSA) serves as a pre-processing such that the main part of the transition is a walk between radial clusterings. In turn, this allowed for the direct use of LP theory in the design of \textsc{AlgoRadToRad}. While the cluster sizes of initial and target clustering are typically very similar, and thus the use of radial clusterings in the transition not a noteworthy restriction, the first and final part of the transition do not change the sites. It would be interesting to study whether a similar walk can be computed without this hard three-part split, where a linear transition from the initial sites to the target sites happens throughout the whole walk. Such a transition could be more `direct', and result in a lower number of steps. There are many other ways that a `better' transition could be designed -- for example, a transition can be influenced by a translation of the data set itself, because the locations of items influence which LSAs are radial and which are not. One could impose additional restrictions of `monotone' cluster sizes during the transition, i.e., cluster sizes may either only increase (if $|C^t_i|> |C^s_i|$) or only decrease (if $|C^t_i| < |C^s_i|$) -- our algorithms in this paper only guarantee that we lie between the given bounds. Further, in some applications a transition where each step consists of the parallel execution of multiple exchanges of items may be desired. For all of these possible improvements, we expect the design of a (combinatorial) algorithm to be challenging, because it does not suffice to use the well-understood relationship between LSAs (and power diagrams) and linear programming over partition and transportation polytopes.

Finally, we worked in a setting in which initial and target clustering are known and separable. Our key goal was to preserve separability throughout. The desire to construct explicit transitions between clusterings can appear in a wide range of settings. In some of these, the way that the target clustering is found may have implications on what a transition should look like. For example, a data set may be clustered as an LSA with outliers, i.e., items that fall out of their cluster's cell. It would be interesting to study a possible generalization of our methods in which outliers, appropriately penalized, are allowed in the transition. Another example would be applications in which the target clustering changes dynamically and frequently, to a point where a `re-optimization' during the transition becomes necessary. And finally, it is possible that the target clustering is not given explicitly, but that only a general goal is stated -- such as ``gradually transition as many customers as possible to fair premium classes in this many steps". Then the search for a target clustering and the computation of the transition become intimately connected.


\section*{Acknowledgment}
\small Borgwardt gratefully acknowledges support of this work through NSF award 2006183 {\em Circuit Walks in Optimization}, Algorithmic Foundations, Division of Computing and Communication Foundations, through AFOSR award FA9550-21-1-0233 {\em The Hirsch Conjecture for Totally-Unimodular Polyhedra}, Airforce Office of Scientific Research, and through Simons Collaboration Grant 524210 {\em Polyhedral Theory in Data Analytics} before. Happach has been supported by the Alexander von Humboldt Foundation with funds from the German Federal Ministry of Education and Research (BMBF).

\small
\section*{Biographies}
Steffen Borgwardt is an Associate Professor in the Department of Mathematical and Statistical Sciences at the University of Colorado Denver. His research lies on the intersection of combinatorial optimization, polyhedral theory, and linear programming. He holds a habilitation on {\em Data Analysis through Polyhedral Theory}, is a lifetime Humboldt fellow, and received a joint EURO Excellence in Practice Award for his work on optimization in land consolidation.

Felix Happach received his Ph.D. in 2020 from the Technische Universität München, as a member of the Operations Research Group that bridges the School of Management and Department of Mathematics. His research interests are in the geometric representation of applied problems from Operations Research and Data Analysis. He received a Master's Thesis Award 2017 of the German Operations Research Society (GOR) and a second place in the Student Paper Prize 2018 of the INFORMS Optimization Society. 

Stetson Zirkelbach is a Ph.D. student at the University of Colorado Denver. He has a background in software development and holds an M.S. in Applied Mathematics, with an emphasis in optimization.

During studies of the combinatorial diameters of partition and transportation polytopes, surprisingly direct applications for the constructed walks came up. They reflected the need for a sequence of clusterings that would retain separability throughout. The research took an algorithmic turn and, through generalizations of edge walks, deeper studies of the geometric properties of the underlying polytopes, and the adaptation of classical linear programming techniques, led to this work.

\normalsize
\bibliographystyle{informs2014}
\bibliography{clusteringpolytopes}

\newpage

\begin{subappendices}
\renewcommand{\thesection}{\Alph{section}}%

\small
\section{Ranging for degenerate vertices}\label{app:ranging}Consider a linear program in standard form given as  
\begin{equation}\label{standardformLP}
\begin{array}{ll@{}ll}
\text{min}  & \displaystyle c^Tx &\\
\text{s.t. }& \displaystyle  Ax = b&\\
				 & x \geq 0.&
\end{array}
\end{equation}
A particularly simple variation of sensitivity analysis is called {\em ranging}: given a change to the objective function $c$ in the form of $c+\lambda \Delta c$ for some $\lambda \in \mathbb{R}$ and some vector $\Delta c$, one has to devise upper (and lower) bounds on $\lambda$ at which the current optimal basis is left. For a nondegenerate optimal vertex $x^*$, information from a simplex tableau for $x^*$ can be used to devise an exact formula; see for example \cite{VanderbeiBook}.

In this paper, we perform ranging for the possibly highly degenerate vertices of $\mathcal{T}^\pm$. A degenerate vertex is represented by multiple bases. We are interested in finding the positive $\lambda$ for when the vertex, not only the current basis, is left. To this end, recall that LP optimality requires primal and dual feasibility. Thus, one can solve an LP over the dual feasible region where strong duality has to remain valid when changing $c$ to $c+\Delta c$. For LP (\ref{standardformLP}), the following LP finds the maximum $\lambda$ such that dual feasibility is retained and strong duality still holds, i.e., the maximum $\lambda$ before leaving the current vertex: 
\begin{equation}\label{dualradialLP}
\begin{array}{ll@{}ll}
\text{max}  & \displaystyle \lambda &\\
\text{s.t. }& \displaystyle  A^T y \leq c + \lambda \cdot \Delta c&\\
				 & b^Ty = c^Tx^*+\lambda \cdot (\Delta c)^Tx^*.&
\end{array}
\end{equation}
\vspace*{1cm}

\section{Radial clusterings for different cluster size bounds}\label{app:radial}
Let us take a brief look at the role of radial clusterings as intermediate steps of the transitions in our approach. A transition to radial clusterings is necessary such that ranging techniques allow for the computation of a linear transition from initial sites $s$ to target sites $t$; see \textsc{AlgoRadToRad}. Radial clusterings, however, are not only LSAs, but optimizers of $c(s)$ over clusterings of all feasible shapes given by $\theta^\pm$ or $\kappa^\pm$. The wider the range of $\theta^\pm$ (or $\kappa^\pm$), the more special these LSAs become. 

\begin{figure}
\centering
\subcaptionbox{A radial clustering (single-shape)}{\includegraphics[width=0.4\textwidth]{5ptSimpleStartLSA.png}}%
\hfill
\subcaptionbox{A radial clustering (all-shape)}{\includegraphics[width=0.40\textwidth]{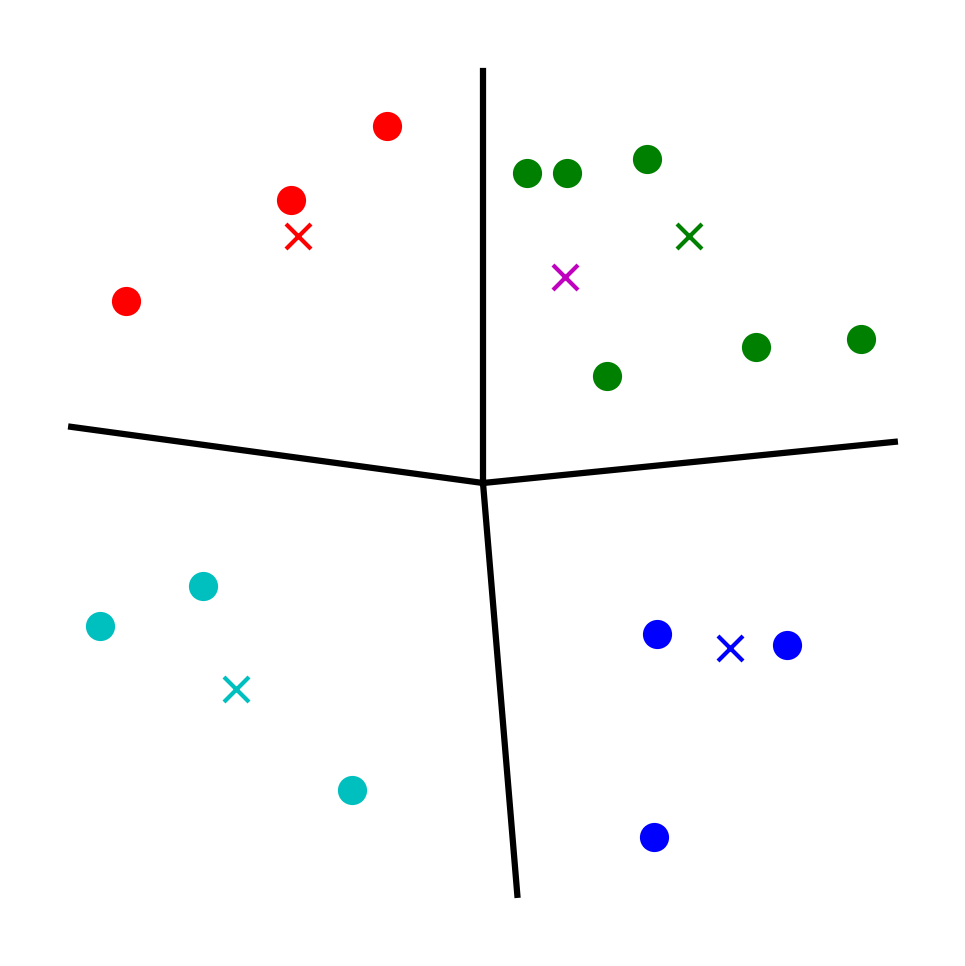}}%
\caption{Two radial clusterings. In (a), the computation was done over $\mathcal{T}^=$, a single-shape transportation polytope. In (b), it was done over $\mathcal{T}$, the all-shape polytope. Note that (a) is a `normal' constrained LSA, while the clusters in (b) are arranged `radially' around the origin, and there is an empty cluster.}
\label{fig:radEx}
\end{figure}

Figure~\ref{fig:radEx} shows two examples. In a single-shape setting, {\em any} LSA is a radial clustering and vice versa, so there is no restriction; see Figure~\ref{fig:radEx}(a). In contrast, in the all-shape setting where all clustering shapes are feasible, items $x_j$ are assigned to the sites $s_i$ (and thus cluster $C_i$) with the largest scalar product $x_j^Ts_i$. Thus, the clusters are arranged `radially' around the origin, and it is even possible to get empty clusters; see Figure~\ref{fig:radEx}(b). In fact, this geometric observation is the reason for the naming of the term. This gives yet another justification for choosing $\theta^\pm=\kappa^\pm$, i.e., for using the smallest range of shapes that allow for a transition. By doing so, no `additional structure' is imposed on the intermediate clusterings: all of them -- radial or not -- are (general) constrained LSAs for prescribed cluster sizes between those of $C^s$ and $C^t$. The only restriction happening in view of our transition is that we do not (need to) consider all constrained LSAs (of other shapes) to construct it. 


\section{Computational Experiments}\label{app:computations}
All examples in this paper were created through a sample implementation and some proof-of-concept computations. An implementation of our algorithms, some examples, and some instructions to run them are available at \url{https://github.com/szirkelbach/Transitioning-Separable-Clusterings}. The implementation is in python and uses Gurobi as an LP solver. We ran our computations on a standard laptop (Intel i7-7700HQ, CPU at 2.80 GHz, 16 GB RAM).

While computational speed is not a primary concern in the applications that we encountered, the scalability of the designed algorithms and the number of steps in the transition are of interest. We briefly report on some observations from computational experiments. We used randomly generated data sets of items in a unit box, as well as random, evenly spread initial and target site vectors, and ran experiments for different numbers $k$ of clusters, $n$ of items, and cluster size bounds. All of the numbers provided in Tables \ref{tab:LSAtoradial} and \ref{tab:radialtoradial} are averages for twenty runs. As we use steps of the simplex method in both \textsc{AlgoLSAtoRad} and  \textsc{AlgoRadToRad}, our current implementations have bottlenecks in the setup and solution of the underlying LPs.

Let us first take a look at \textsc{AlgoLSAtoRad}, the computation of a radial clustering from an LSA. Recall that LSAs and radial clusterings are equivalent if cluster sizes are fixed. When generating an instance for our computations, we randomly generated cluster size bounds between $(1-\epsilon)\cdot \frac{n}{k}$ and $(1+\epsilon)\cdot \frac{n}{k}$ for various values of a {\em cluster size slack} $\epsilon$. In Table \ref{tab:LSAtoradial}, we report on our observations. The first three columns represent the scale of the problem: the number $k$ of clusters and $n$ of items, and the resulting matrix size for the underlying LP, i.e., the number of coefficients in the constraint matrix of Formulation~(\ref{transport}), derived as $(2k+n)\cdot kn$. The cluster size slack $\epsilon$ is represented as a percentage of $\frac{n}{k}$. The final columns show the number of steps and total time for different values of $\epsilon$.

 \begin{table}[b]
     \centering
     \begin{tabular}{|c|c|c|c|c|c|c|c|c|}\hline 
       & & & \multicolumn{3}{c|}{number of steps} & \multicolumn{3}{c|}{total time (in sec)}   \\ \hline
          k&n & matrix size& $\epsilon =10\%$ & $\epsilon =25\%$ & $\epsilon =40\%$ & $\epsilon =10\%$ & $\epsilon =25\%$ & $\epsilon =40\%$\\ \hline
          5 & 100 & 55,000 & 1.55 & 5.35 & 8.56 & 0.01 & 0.04 & 0.06 \\ \hline
		 10 & 100 & 120,000 & 1.60 & 5.88 & 11.8 & 0.01 & 0.12 & 0.23\\ \hline
 		 10 & 500 & 2,600,000 & 12.5 & 24.0 & 50.2 & 0.96 & 3.07 & 4.87 \\ \hline
		 10 & 1000 & 10,200,000 & 24.6 & 70.2 & 132 & 4.88 & 11.1 & 19.6 \\ \hline
		 10 & 5000 & 251,000,000 & 157 & 380 & 625 & 180 & 377 & 579 \\ \hline
		 20 & 5000 & 504,000,000 & 129 & 335 & 640  & 213 & 659 & 1,260 \\ \hline
     \end{tabular}
     \caption{Computational experiments for \textsc{AlgoLSAtoRad}, the transition from an LSA to a radial clustering.}
     \label{tab:LSAtoradial}
 \end{table}

We were able to implement \textsc{AlgoLSAtoRad} as a single run of primal simplex that stays in memory. Dynamic updates to the right-hand sides and optimality checks for different cluster size bounds are performed. As such, we observe low computation times until the underlying matrix becomes very large.  Essentially, performance and scaling of our algorithm are similarly behaved to solving a linear program, with some overhead due to data processing and entering the optimization process for updates. The number of steps in the transition increases about linearly with the number of items and more than linearly with $\epsilon$. Note that a large $\epsilon$ (such as $40\%$) allows for a dramatic difference in clustering shapes. In practice, we only encountered low values of $\epsilon$ and only few steps were required. 

Next, we turn to \textsc{AlgoRadToRad}, the transition between radial clusterings. The number of steps in such a transition depends on the `distance' between the initial and target sites or the `difference' between initial and target radial LP. This information can be represented in multiple ways. To have a measure that takes into account both sites and data set, we distinguish our runs by the percentage $\delta$ of items assigned to different clusters. For example $25\%$ for $n=2,500$ would refer to $625$ items assigned to different clusters. In Table \ref{tab:radialtoradial}, we report on our observations. The first three columns again represent the scale of the problem. For \textsc{AlgoRadToRad}, the matrix size is relevant in two ways: it is not only the size of Formulation~(\ref{transport}), but also roughly the size for LP (\ref{dualradialLP}) used in the breakpoint computation. For the different problem sizes, we show the average time per iteration, i.e., for the computation of a breakpoint, as well as overall number of steps and total time for $\delta=10\%$ and $\delta=25\%$. 

 \begin{table}[t]
     \centering
     \begin{tabular}{|c|c|c|c|c|c|c|c|}\hline 
       & & & time per iteration & \multicolumn{2}{c|}{number of steps} & \multicolumn{2}{c|}{total time (in sec)}   \\ \hline
          k&n & matrix size & (in sec) & $\delta =10\%$ & $\delta =25\%$  & $\delta =10\%$ & $\delta =25\%$ \\ \hline
          5 & 100 & 55,000 & $<0.1\;\;$ & 6 & 13 & 0.2& 0.4\\ \hline
		 10 & 100 & 120,000 & $<0.1\;\;$ & 6 & 12 & 0.2 & 0.5 \\ \hline
 		 10 & 500 & 2,600,000 & 0.2 & 29 & 68 & 5.7 & 14.0   \\ \hline
		 10 & 1000 & 10,200,000 & 0.7 & 58 & 124 & 39.9 & 87.5 \\ \hline
		 10 & 2500 & 63,000,000 & 6.2 & 121 & 380 & 747 & 2,364 \\ \hline
		 20 & 2500 & 127,000,000 & 16.1 & 104 & 345 & 1,675 & 5,554 \\ \hline
		 10 & 5000 & 251,000,000 & 34.5 & 232 & 702 & 8,007 & 24,205  \\ \hline
		 20 & 5000 & 504,000,000 & 113.2 & 220 & 685  & 24,794 & 77,655  \\ \hline
     \end{tabular}
     \caption{Computational experiments for \textsc{AlgoRadToRad}, the transition between radial clusterings.}
     \label{tab:radialtoradial}
 \end{table}

There are a couple of interesting observations. First, the overall computation times are significantly larger than for \textsc{AlgoLSAtoRad} and this effect becomes more pronounced for larger data sets. Because of this, it is appropriate to think of \textsc{AlgoLSAtoRad} as a pre-processing step for \textsc{AlgoRadToRad}. 

In each iteration of \textsc{AlgoRadToRad}, LP (\ref{dualradialLP}) is solved to find a breakpoint and the clustering is updated. Both right-hand side and some coefficients of this LP change every time, so while we are able to keep it in memory and update it dynamically, a full solution is required in the current implementation. Additionally, LP (\ref{dualradialLP}) has complicated coefficients $\Delta c$ in the constraint matrix associated to the breakpoint variable $\lambda$ (and for $c^Tx^*$), so even at the same matrix size, the LP is harder to solver than an LP over Formulation~(\ref{transport}), which has only zeroes and ones. For increasing problem sizes, Gurobi uses more and more time to verify correctness of its intermediate results to guarantee numerical stability. We added a few extra rows to Table \ref{tab:radialtoradial} compared to Table \ref{tab:LSAtoradial} to exhibit where this effect becomes noticeable. In contrast, the update to the new clustering through a few primal simplex steps over Formulation~(\ref{transport}), kept in memory and warm started with a vertex that is already adjacent to the new optimum, is negligible. 

It is important to note that different values of $\delta$ do not impact the size of underlying LPs and thus the average time per iteration. (Note that the dimension of a data set only goes into the computation of distances and does not impact the size of the LPs either.) Memory requirements remain constant throughout. Despite long computation times and a large number of steps for large problem sizes, they are readily solvable on a normal laptop. The scaling of \textsc{AlgoRadToRad} is limited not by memory, but issues of numerical stability in the LP solver and time available.

However, of course different values of $\delta$ strongly affect the total number of steps. We observed a roughly linear dependence of the number of steps on $\delta$, independently of the number of clusters. For the number of clusters itself, we did not see a clear effect. When the number of clusters was higher, longer exchanges of items were constructed in the iterations and the total number of steps remained similar (or became slightly lower). In all of our runs, the number of steps remained below the total number of items to be moved.

\end{subappendices}
\end{document}